\documentclass[11pt,a4paper,reqno]{amsart}
\usepackage{amssymb,a4wide,amsthm,amsfonts}
\usepackage[latin1]{inputenc}
\usepackage[T1]{fontenc}
\usepackage{lmodern}
\usepackage{enumitem}
\usepackage{color}
\usepackage{mparhack}
\usepackage{graphicx}
\usepackage{mathrsfs}
\usepackage{pgfplots}
\usepackage{subcaption}
\usepackage[firstinits=true,doi=false,isbn=false,url=false,maxnames=7,backend=bibtex]{biblatex}
\numberwithin{equation}{section}
\usepackage[]{hyperref}
\usepackage{todonotes}

\usepgfplotslibrary{external} 
\pgfplotsset{compat=newest}

\hypersetup{colorlinks=false,bookmarksnumbered=true,plainpages=false,linkcolor=dunkelblau,urlcolor=dunkelblau,citecolor=dunkelblau,pdfstartview=FitH,%
pdftitle={Optimal control for the thin-film equation: Convergence of a multi-parameter approach to track state constraints avoiding degeneracies},%
pdfauthor={Markus Klein},%
pdfcreator={LaTeX},%
}

\theoremstyle{plain}
\newtheorem{thm}{Theorem}[section]
\newtheorem{lemma}[thm]{Lemma}

\newtheorem{prob}[thm]{Problem}

\newtheorem{alg}[thm]{Algorithm}

\theoremstyle{remark}

\newcommand{\R}{\mathbb{R}} 
\newcommand{\Con}{\mathcal{C}} 
\newcommand{\ve}{\varepsilon} 
\newcommand{\dd}{\,\mathrm{d}} 

\DeclareMathOperator{\supp}{supp} 

\DeclareMathOperator{\interior}{int}
\DeclareMathOperator{\id}{id}

\pgfplotsset{two_style/.style={
width=\textwidth,
height=0.25\textheight,
legend style={font=\footnotesize,inner sep=2pt,nodes={inner sep=2pt,text depth=0.1em}},
tick label style={font=\footnotesize},
label style={font=\small},
title style={font=\small},
max space between ticks={40},
},
three_style/.style={
width=\textwidth,
height=0.18\textheight,
legend style={font=\footnotesize,inner sep=2pt,nodes={inner sep=2pt,text depth=0.1em}},
tick label style={font=\scriptsize},
label style={font=\small},
title style={font=\small},
max space between ticks={40},
axis y line=left,
axis x line=middle,
try min ticks={4},
}
}

\allowdisplaybreaks

\begin{document}
\raggedbottom 

\title[Optimal control for the thin film equation]{Optimal control for the thin film equation: Convergence of a multi parameter approach to track state constraints avoiding degeneracies}
\date{\today}

\author{Markus Klein}
\address{Mathematisches Institut, Universität Tübingen, Auf der Morgenstelle 10, 72076 Tübingen, Germany.}
\email{klein@na.uni-tuebingen.de}

\author{Andreas Prohl}
\address{Mathematisches Institut, Universität Tübingen, Auf der Morgenstelle 10, 72076 Tübingen, Germany.}
\email{prohl@na.uni-tuebingen.de}

\thanks{The work supported by a DFG grant within the Priority Program SPP 1253 (Optimization with Partial Differential Equations). This work was performed on the computational resource bwUniCluster funded by the Ministry of Science, Research and Arts and the Universities of the State of Baden-Württemberg, Germany, within the framework program bwHPC; cf.~\cite{bwgrid}. The authors are grateful to the anonymous referees for their valuable comments, which helped to improve the quality of this article.}

\begin{abstract}
We consider an optimal control problem subject to the thin-film equation. The PDE constraint lacks well-posedness for general right-hand sides due to possible degeneracies; state constraints are used to circumvent this problematic issue and to ensure well-posedness. Necessary optimality conditions for the optimal control problem are then derived. A convergent multi-parameter regularization is considered which addresses both, the possibly degenerate term in the equation and the state constraint. Some computational studies are then
reported  which evidence the relevant role of the state constraint, and motivate proper scalings of involved regularization and numerical parameters.
\end{abstract}

\maketitle

\section{Introduction}

Let $\Omega=(a,b)\subset \R$, $0<T<\infty$, let $\lambda, C_0 >0$, $u\in L^2(H_0^1)$, and $H^2(\Omega)\ni 
y_0 \geq  C_0$ be given. The one-dimensional thin film equation with external control $u:\Omega_T \rightarrow {\mathbb R}$ reads as follows: Find $y:\Omega_T\to \R$ such that
($\beta > 1$)
\begin{align}
y_t = -(\lambda \vert y\vert^{\beta}y_{xxx})_x + u_x, \label{eq:thin_film}
\end{align}
together with the initial condition $y(0,.)=y_0$ and boundary conditions $y_x=y_{xxx}=0$ in $a,b$ for $0<t<T$. 

In this work, we study the following constrained optimization problem related to \eqref{eq:thin_film}.
\begin{prob} \label{prob:thin_film_opti}
Let $\tilde y\in L^2(\Omega_T)$ be given, $\alpha>0$, and $C_0>0$. Find a minimum $(y^*,u^*)\in L^2(H^4)\cap H^1(L^2)\times L^2(H_0^1)$ of
\[ J(y,u):= \frac12\int_0^T \int_{\Omega} \vert y-\tilde y\vert^2 \dd x \dd t + \frac{\alpha}2 \int_0^T \int_{\Omega} \vert u_x\vert^2 \dd x \dd t \]
subject to \eqref{eq:thin_film} and $y\geq C_0$ in $\Omega_T$. 
\end{prob}

The aim of this problem is to control the height $y$ of a fluid film via an external map $u_x:\Omega_T\to \R$,
while respecting a state constraint to prevent its rupture (i.e., $y(t,x) = 0$). The governing equation \eqref{eq:thin_film} is in divergence form to avoid evaporation or wetting effects. 
A~possible application of the above optimal control Problem~\ref{prob:thin_film_opti} is in the fabrication of electronic chips, where thin layers of different material are deposited on a Si wafer \cite{becker_nature}. For an efficient electronical circuit, each layer has to constitute a specific profile $\tilde{y}$ where the material should be deposited. The problem is to find external forces such that the solution of \eqref{eq:thin_film} is near this desired profile $\tilde y$.

Equation \eqref{eq:thin_film} is e.g.~derived in \cite{becker_gruen_lenz_rumpf,bertozzi_ams,oron1997}, see also
\cite{bertozzi_ams,gruen_entropy}: we consider the fluid layer to be thin, i.e., $\tau:=\text{height}/\text{length}\ll 1$. A non-dimensional transformation from the classical Navier--Stokes equation which is based on the small ratio $\tau$ and a Taylor expansion of the terms, together with the assumption of a so-called no slip boundary condition  (cf.~\cite[p. 936]{oron1997}) leads to an asymptotic expansion in~$\tau$. Neglecting higher order terms of $\tau$, and the proper use of boundary conditions then leads to \eqref{eq:thin_film}. 
We note that through the transformation process, a conservative force on the right-hand side of the Navier--Stokes equation transforms into an additional term $-(g_0(y)y_x)_x$ on the right-hand side in \eqref{eq:thin_film}, where $g_0$ denotes a potential function (see \cite{becker_gruen_lenz_rumpf} for details). Hence, a control problem for the Navier--Stokes equation with distributed conservative force (cf.~\cite{abergel_temam}) transforms `naturally' into an optimal control problem for the thin film equation where a potential function $g_0$ is to be found, instead of searching for a $L^2(H^1_0)$ control function $u$ as in Problem~\ref{prob:thin_film_opti} for \eqref{eq:thin_film}. 
However, it is not clear how to deal with such a problem in general, and there are only a few works in this direction which deal with inverse problems, see e.g., \cite{kaltenbacher_klibanov}. The authors are not aware if and how the methods which are used there can to transformed to more complicated scenarios such as \eqref{eq:thin_film}. Also, we are not aware how to practically construct such a potential function unless the potential function is specified to belong to a specific class of functions (e.g., polynomials, or a sum of other given potentials). Therefore, we will neglect such an $g_0$-term in this work.

The fundamental analytical work for the PDE \eqref{eq:thin_film} with $u\equiv 0$ is \cite{bernis_friedman}, where uniform
bounds in terms of the energy $E[y] = \frac{1}{2} \int_{\Omega} \vert y_x\vert^2 \, {\rm d}x$ for a solution of a regularized version of
\eqref{eq:thin_film} are used to construct a weak solution of \eqref{eq:thin_film}; cf.~\cite[Theorem 3.1]{bernis_friedman}.  Positivity of solutions of \eqref{eq:thin_film} with $u = 0$ for $\beta \geq 4$ is then
shown by a contradiction argument which uses the entropy functional $H[y] = \frac{1}{(\beta-1)(\beta-2)} \int_{\Omega} y^{2-\beta}\, {\rm d}x$; see~\cite[Theorem 4.1]{bernis_friedman}. {We briefly recapitulate this argument to show Lemma~\ref{lemma:thin_film_g0_gleich_0_positivity}, which later then settles 
solvability of Problem~\ref{prob:thin_film_opti}; see Theorem \ref{thm:thin_film_existence_opti_mod}.}

Our main goal is then to derive necessary optimality conditions for Problem~\ref{prob:thin_film_opti}. For this purpose,  we need to modify some of the proofs in \cite{bernis_friedman} to properly address  equation (\ref{eq:thin_film}) with general right-hand side $u_x$. 
In particular, it is the possibly degenerate character of the PDE (\ref{eq:thin_film}) which affects some energy arguments,
while the entropy argument mentioned above is not known to be valid
any more for $u \neq 0$; see also
Figure~\ref{fig:thin_film_negative_sol_only_state_bsp}, where an approximate solution taking values in $\R$ to a given non-trivial external force $u_x$ is displayed.

{This problematic issue of PDE (\ref{eq:thin_film}) is also apparent in
Problem~\ref{prob:thin_film_opti} ({\em in the absence of the state constraint}). To avoid it,} one strategy could be to only take into account those exterior forces $u_x$ for which a corresponding solution $y$ exists. Unfortunately, we cannot give a good characterization of such a set of controls, and we do not know topological properties of it. There are a few recent articles dealing with different degenerate optimal control problems which share this problematic issue; see e.g.,~\cite{clason_kaltenbacher_westervelt,clason_kaltenbacher_comp,clason_kaltenbacher_mems}.
A second strategy to overcome this problem is to ensure the strict positivity of solutions to 
\eqref{eq:thin_film}  by a state constraints as stated in Problem~\ref{prob:thin_film_opti}. General 
controls $u$ are admitted in this case, but to derive necessary optimality conditions becomes more involved.

We refer the reader to \cite{clason_kaltenbacher_comp} where both strategies are compared for a different equation, and the second scenario is given preference  to cope with possible degeneracies arising in the governing equation. It is argued that the set of external forces in the first scenario may not be rich enough, and therefore possible target profiles $\tilde{y}$ may not be reached. It is also in this work that we prefer the second strategy to study Problem~\ref{prob:thin_film_opti}.

Our first result is solvability of the optimal control Problem~\ref{prob:thin_film_opti}. Then, we use an abstract result for state constrained optimization problems from~\cite{alibert_raymond} to derive necessary optimality conditions for Problem~\ref{prob:thin_film_opti}. In order to overcome technical difficulties in the proofs in section \ref{sec:thin_film_opti_reg}, we need controls $u \in L^2(H_0^1)$ in \eqref{eq:thin_film}, which is already stated in Problem~\ref{prob:thin_film_opti}. This implies $y \in L^2(H^4)\cap H^1(L^2)$ for states, which is then sufficient to construct Lagrange multipliers in proper spaces.

The optimality conditions \eqref{eq:thin_film_opt_cond} involve non-regular Lagrange multipliers in the dual space of $L^2(H^4)\cap H^1(L^2)$, which hinders an immediate numerical treatment: a typical strategy to overcome this problem is to relax the state constraint $y\geq C_0$ by penalty approximation~\cite{bryson_buch}, Moreau-Yosida approximation \cite{hintermueller_kunisch_my}, or mixed control-state constraints (Lavrentiev regularization)~\cite{lavrentiev_reg}. We recall that the state constraint is not an additional requirement on admissible states, but is essential for the well-posedness of the PDE \eqref{eq:thin_film}.

We propose the following strategy to ensure well-posedness of \eqref{eq:thin_film} in the context of relaxation methods:\begin{itemize}
\item We establish regularity results for the {\em regularized} PDE (\ref{eq:thin_film_eps}) (section~\ref{sec:thin_film_state_equation_fixed_rhs}), then show solvability of Problem~\ref{prob:thin_film_opti}
and derive necessary optimality conditions (section~\ref{sec:thin_film_opt_original}).
\item In section~\ref{sec:thin_film_opti_reg}, we study the optimal control Problem~\ref{prob:thin_film_opti_thin_film_eps} which uses the regularized state equation (\ref{eq:thin_film_eps}) for $\varepsilon >0$. 
In particular, related optimality conditions (\ref{eq:thin_film_opt_cond_eps}) are derived, and convergence
to the necessary optimality conditions \eqref{eq:thin_film_opt_cond} of the original Problem \ref{prob:thin_film_opti} is shown. For this purpose, it is relevant to bound the norm of $u_x$ as it is given in the functional, which helps to properly bound all corresponding Lagrange multipliers. 
\item In section~\ref{sec:thin_film_penalty}, we show convergence of the penalty approach (Problem~\ref{prob:thin_film_opti_thin_film_eps_gamma}) towards Problem~\ref{prob:thin_film_opti_thin_film_eps} for $\gamma \rightarrow 0$.
Since the equality constraint in Problem~\ref{prob:thin_film_opti_thin_film_eps_gamma} is well-posed for every $\ve>0$, we may use a standard numerical approach to solve the corresponding optimality conditions \eqref{eq:thin_film_opt_cond_eps_gamma}. Corresponding computational studies are reported in section~\ref{sec:thin_film_comp} which study proper balancing of finite parameters $\varepsilon, \gamma >0$ to numerically
solve the necessary optimality conditions for Problem~\ref{prob:thin_film_opti_thin_film_eps_gamma}.
\end{itemize}

We emphasize the relevancy to study the intermediate optimization Problem~\ref{prob:thin_film_opti_thin_film_eps} since it is not clear how to simultaneously tend both regularization parameters to zero. In particular, the parameter $\gamma>0$ to regularize the state constraint is the first which tends to zero; it is here that we benefit from the well-posedness of the involved equality constraint for every $\ve>0$ to construct a solution of the intermediate optimization Problem~\ref{prob:thin_film_opti_thin_film_eps}. 

To our knowledge, this is the first work which deals with optimal control subject to the thin film equation. A regularization of Problem \ref{prob:thin_film_opti_thin_film_eps} ($\ve=1$) is studied 
in \cite{zhao_liu_thin_film}, which coincides with the intermediate optimal control problem  but without state constraints. However, existence for the limiting problem related to Problem~\ref{prob:thin_film_opti}, and a convergence analysis for $\ve\to0$ are left open in \cite{zhao_liu_thin_film}.

Throughout this article, we use the following notation: We write $\Vert.\Vert$ for the $L^2(\Omega)$ or $L^2(\Omega_T)$-norm when it is clear from the context that we only integrate in space or both, in space and time. Let $W^{k,p}$ and $H^k:=W^{k,2}$ denote standard Sobolev spaces. By 
\[ W^{k,p}(W^{m,q}):=W^{k,p}(0,T;W^{m,q}) \]
we refer the reader to standard Bochner spaces. The space $\Con$ ensembles continuous functions, while $\Con^{0,\alpha}$ denotes corresponding Hölder spaces. 

The dual pairing of $X$ and its dual space $X^*$ is denoted as $\langle .,.\rangle_{X,X^*}$. For the scalar products in $L^2$ and $L^2(L^2)$,  we write $(\cdot, \cdot)$ at places where no ambiguities arise; otherwise, we add the corresponding spaces as index to the scalar product.

We use $C$ as a generic non-negative constant; to indicate dependencies, we write $C(.)$.

\section{The regularized state equation} \label{sec:thin_film_state_equation_fixed_rhs}

We show properties of solutions of a regularization of the equation \eqref{eq:thin_film}. 
The arguments in this section adapt corresponding ones in \cite{bernis_friedman}.

\begin{prob}
Let $\lambda>0$, $\ve\geq 0$. Find $y:\Omega_T\to \R$ such that
\begin{align}
y_t = -\big( [\lambda \vert y\vert^\beta+\ve]y_{xxx}\big)_x + u_x, \label{eq:thin_film_eps}
\end{align}
together with initial condition $y(0)=y_0\in H^2(\Omega)$ and boundary conditions $y_x=y_{xxx}=0$ in $a,b$.
\end{prob}

\subsection{Regularity and properties of solutions}

\begin{lemma} \label{lemma:thin_film_equation_mass_conservation}
Let $\lambda>0$, $\ve\geq 0$, $u\in L^2(H_0^1)$, and $y$ be a weak solution of \eqref{eq:thin_film_eps}. Then, the mass is conserved, i.e.,
\begin{align}
\int_{\Omega} y(t,.)\dd x = \int_{\Omega} y_0 \dd x \quad \forall \,0\leq t\leq T. \label{eq:thin_film_massenerhaltung}
\end{align}
\end{lemma}
\begin{proof}
Integrate \eqref{eq:thin_film} over $\Omega$ and use the divergence theorem together with the boundary conditions for $u$ to prove \eqref{eq:thin_film_massenerhaltung}.
\end{proof}
In the following, let $E[y]  = \frac{1}{2} \int_{\Omega} \vert y_x\vert^2\, {\rm d}x$.
\begin{lemma} \label{lemma:thin_film_hoelder_space}
Let $\lambda>0$, $\ve\geq 0$,  $u\in L^2({ H^1_0})$, and $y:\Omega_T\to \R$ be a {weak} solution of \eqref{eq:thin_film_eps} with $y\geq C_0 > 0$ in $\Omega_T$. Then there exists a constant $C = C( T,C_0, E[y_0],\Vert u\Vert_{L^2(L^2)}) >0$ such that 
\begin{align}
\sup_{t \geq 0} E[y(t)]  + (\lambda C_0^\beta+\ve)\Vert y_{xxx}\Vert_{L^2(L^2)}^2 \leq C. \label{eq:thin_film_eps_apriori_weak}
\end{align}
In particular, $y$ is Hölder continuous in space, i.e., there exists a constant $H_{\text{\rm space}}>0$ such that
\[ \vert y(t,x_1)-y(t,x_2)\vert \leq H_{\text{\rm space}}\vert x_1-x_2\vert^{\frac12} \quad \forall \,0\leq t\leq T, \; x_1,x_2\in\Omega. \]
\end{lemma}
\begin{proof}
We multiply \eqref{eq:thin_film_eps} with $-y_{xx}$, integrate over $\Omega$, and arrive for almost all $t\in [0,T]$ at
\begin{align}
\frac12\frac{d}{dt}\Vert y_x\Vert^2 + \int_{\Omega} (\lambda\vert y\vert^\beta+\ve)y_{xxx}^2 \dd x = - \int_{\Omega} u_x y_{xx} \dd x =: I. \label{eq:thin_film_hoelder_space_nach_testen}
\end{align}
For $\sigma>0$, the term $I$ can be estimated by
\[ I = \int_{\Omega} u y_{xxx} \dd x \leq \sigma\Vert y_{xxx}\Vert^2 + C(\sigma)\Vert u\Vert^2. \]
Using that $\lambda\vert y\vert^\beta+\ve\geq \lambda C_0^\beta+\ve$, choosing $\sigma$ sufficiently small, and finally using Gronwall's inequality validates the lemma. The Hölder continuity follows by one-dimensional Sobolev embeddings.
\end{proof}

\begin{lemma} \label{lemma:thin_film_hoelder_time}
Let $\lambda>0$, $\ve\geq 0$,  $u\in L^2(L^2)$, and $y:\Omega_T\to \R$ be a weak solution of \eqref{eq:thin_film_eps} with $y\geq C_0>0$ in $\Omega_T$. Then there exists a constant $H_{\text{\rm time}} = H_{\text{\rm time}}(T,C_0, E[y_0],\Vert u\Vert_{L^2(L^2)})>0$ such that
\[ \vert y(t_2,x)-y(t_1,x)\vert \leq H_{\text{\rm time}}\vert t_2-t_1\vert^{\frac18} \quad \forall\, 0\leq t_1,t_2\leq T, \; x\in\Omega. \]
\end{lemma}
\begin{proof}
The proof uses arguments similar (for $u = 0$) to those given in \cite[Lemma 2.1]{bernis_friedman}.

{\em Step 1:} Assume the statement is not correct. Then for every $M>0$ there exist $x_0\in\Omega$ and $0\leq t_1,t_2\leq T$ such that
\begin{align} 
\vert y(t_2,x_0)-y(t_1,x_0)\vert > M \vert t_2-t_1\vert^{\kappa} \label{eq:thin_film_hoelder_time_contradiction_1}
\end{align}
for $\kappa=\frac18$. Without restriction let us assume that $t_1<t_2$ and $y(t_2)>y(t_1)$. Then \eqref{eq:thin_film_hoelder_time_contradiction_1} reads as
\begin{align}
y(t_2,x_0)-y(t_1,x_0) > M (t_2-t_1)^{\kappa}. \label{eq:thin_film_hoelder_time_contradiction_2}
\end{align}
In the proof, we will show that $M$ can be uniformly bounded with respect to $x_0,t_1$ and $t_2$, which contradicts \eqref{eq:thin_film_hoelder_time_contradiction_2}.

We construct an appropriate test function for the equation \eqref{eq:thin_film_eps}. Let
\[ \xi(x) := \xi_0\left( \frac{x-x_0}{\frac{M^2}{16H_{\text{space}}^2}(t_2-t_1)^{2\kappa}}\right), \]
where $M$ is from \eqref{eq:thin_film_hoelder_time_contradiction_2}, and $H_{\text{space}}$ from Lemma~\ref{lemma:thin_film_hoelder_space}. The function $\xi_0\in \Con_0^{\infty}$ has the properties $\xi_0(x)=\xi_0(-x)$, $\xi_0(x) := 1$ for $0\leq x<\frac1{2L}$ for some $L>0$ ($L$ will be chosen later and will only depend on $H_{\text{space}}>0$ from Lemma~\ref{lemma:thin_film_hoelder_space} and on $\Omega$), $\xi_0(x) := 0$ for $x\geq 1$ and $\xi_0'(x)\leq 0$ for $x\geq 0$. In particular, we have
\[ \xi(x) = \begin{cases} 0, & \vert x-x_0\vert\geq \frac{M^2}{16H_{\text{space}}^2}(t_2-t_1)^{2\kappa}, \\ 1, & \vert x-x_0\vert\leq \frac1{2L} \frac{M^2}{16H_{\text{space}}^2}(t_2-t_1)^{2\kappa}. \end{cases} \]
We define the function $\theta_{\delta}$ by
\[ \theta_{\delta}(t):= \int_{-\infty}^t \theta_{\delta}'(s) \dd s, \]
where
\[ \theta_{\delta}'(t) = \begin{cases} \frac1{\delta}, & \vert t-t_2\vert<\delta, \\ -\frac1{\delta}, & \vert t-t_1\vert<\delta, \\ 0, & \text{ else} \end{cases} \]
for $0<\delta<\min\lbrace\frac12(t_2-t_1),t_1,T-t_2\rbrace$ small enough.

We consider the function $\phi(t,x):=\xi(x)\theta_{\delta}(t)$, multiply \eqref{eq:thin_film_eps} with $\phi$, integrate over $\Omega_T$ and get
\begin{align}
\int_0^T y\phi_t \dd x \dd t = - \int_0^T \int_{\Omega} (\lambda\vert y\vert^\beta+\ve)y_{xxx} \phi_x \dd x \dd t + \int_0^T \int_{\Omega} u\phi_x \dd x \dd t. \label{eq:thin_film_hoelder_time_test_equation}
\end{align}

{\em Step 2:}
We derive a lower bound for the left-hand side of \eqref{eq:thin_film_hoelder_time_test_equation}. By the construction of $\theta_{\delta}$, its time derivative approximates like a Dirac function evaluated at $t_1$ and $t_2$, respectively. More precisely, we have for $\delta\to 0$
\begin{align}
\int_0^T \int_{\Omega} y(t,x)\xi(x)\theta_{\delta}'(t) \dd x \dd t \to \int_{\Omega} \xi(x) \big[ y(t_2,x)-y(t_1,x)\big] \dd x. \label{eq:thin_film_hoelder_time_neuer_verweis}
\end{align}
We consider points $x$ such that
\begin{align} 
\vert x-x_0\vert\leq\frac{M^2}{16H_{\text{space}}^2}(t_2-t_1)^{2\kappa} \label{eq:thin_film_hoelder_time_estimate_x_x_0}
\end{align}
since the corresponding integral in \eqref{eq:thin_film_hoelder_time_test_equation} outside this ball vanishes. For such $x$, there holds by \eqref{eq:thin_film_hoelder_time_contradiction_2} and Lemma~\ref{lemma:thin_film_hoelder_space}
\begin{align*}
y(t_2,x)-y(t_1,x) &= \big[y(t_2,x)-y(t_2,x_0)\big] + \big[ y(t_2,x_0)-y(t_1,x_0)\big] + \big[y(t_1,x_0)-y(t_1,x)\big] \\
&\geq -2H_{\text{space}}\vert x-x_0\vert^{\frac12} + M(t_2-t_1)^{\kappa} \geq \frac{M}2(t_2-t_1)^{\kappa},
\end{align*}
where we also used \eqref{eq:thin_film_hoelder_time_estimate_x_x_0}. For $L=L(\Omega,H_{\text{space}})>0$ appropriate, we have $\lbrace \xi=1\rbrace\subset \Omega$. We may estimate the term in \eqref{eq:thin_film_hoelder_time_neuer_verweis} from below as follows,
\begin{align} 
\int_{\Omega} \xi(x)\big[ y(t_2,x)-y(t_1,x)\big] \dd x \geq \frac{M}2(t_2-t_1)^{\kappa} \frac1{2L}\frac{M^2}{16H_{\text{space}}^2}(t_2-t_1)^{2\kappa} = CM^3(t_2-t_1)^{3\kappa}. \label{eq:thin_film_hoelder_time_lhs_lower_bound}
\end{align}
{\em Step 3:} We derive an upper bound for the first term on right-hand side of \eqref{eq:thin_film_hoelder_time_test_equation}.
\begin{multline*}
\int_0^T \int_{\Omega} (\lambda\vert y\vert^\beta+\ve)y_{xxx} \phi_x \dd x \dd t \\
\begin{aligned}
&\leq \Vert \lambda\vert y\vert^\beta+\ve \Vert_{L^{\infty}(\Omega_T)} \Vert y_{xxx}\Vert_{L^2(L^2)} \left( \int_0^T \int_{\Omega} [\xi'(x)]^2[\theta_{\delta}(t)]^2 \dd x \dd t \right)^{\frac12} \\
&\leq \Vert \lambda\vert y\vert^\beta+\ve \Vert_{L^{\infty}(\Omega_T)} \Vert y_{xxx}\Vert_{L^2(L^2)} \overbrace{\left( \int_{\Omega} [\xi'(x)]^2\dd x \right)^{\frac12}} \underbrace{\left( \int_0^T [\theta_{\delta}(t)]^2 \dd t \right)^{\frac12}} \\
&\leq C(H_{\text{space}}) \overbrace{\frac1{\frac{M^2}{16H_{\text{space}}^2}(t_2-t_1)^{2\kappa}}\Vert \xi_0'\Vert_{L^{\infty}(\Omega)} \frac{M}{4H_{\text{space}}} (t_2-t_1)^{\kappa}}\underbrace{2(t_2-t_1+2\delta)^{\frac12}},
\end{aligned}
\end{multline*}
where we used that the first two norms are uniformly bounded via Lemma~\ref{lemma:thin_film_hoelder_space} by $C(H_{\text{space}})$. The factor $\frac{M}{4H_{\text{space}}}(t_2-t_1)^{\kappa}$ is the integral of $1$ over $\supp \xi$, while $(t_2-t_1+2\delta)^{\frac12}$ is the Lebesgue measure of the support of $\theta_{\delta}$, where we use that $\theta_{\delta}$ is uniformly bounded by $2$ (We highlight the affiliation of each term in the last estimate). Note that the constant $C$ depends on $H_{\text{space}}$ from Lemma~\ref{lemma:thin_film_hoelder_space} (i.e., on $T,C_0, E[y_0]$, and $\Vert u\Vert_{L^2(L^2)}$), but it does not depend on $\ve$, $M$ or $\delta$.

{\em Step 4:} We estimate the second term in \eqref{eq:thin_film_hoelder_time_test_equation},
\[ \int_0^T \int_{\Omega} u\phi_x \dd x \dd t \leq \Vert u\Vert\Vert \phi_x\Vert \leq C(H_{\text{space}}) \frac1M(t_2-t_1)^{-\kappa}(t_2-t_1+2\delta)^{\frac12}. \]

{\em Step 5:} For $\delta\to 0$, we get at the end
\[ M^3(t_2-t_1)^{3\kappa} \leq C\frac1M(t_2-t_1)^{\frac12-\kappa}, \]
where the constant $C$ is independent of $x_0,t_1,t_2$ and $M$. This leads to $M\leq \sqrt[4]{C}$, which contradicts \eqref{eq:thin_film_hoelder_time_contradiction_2}, and the lemma follows. 
\end{proof}

{ \begin{lemma} \label{lemma:thin_film_apriori_h4}
Let $\lambda>0$, $\ve\geq 0$,  $u\in L^2(H_0^1)$, and $y:\Omega_T\to \R$ be a weak solution of \eqref{eq:thin_film_eps} with $y\geq C_0$ in $\Omega_T$. There exists $C = C(C_0, \beta, T, \Vert u_x\Vert_{L^2(L^2)}) > 0$ such that
\begin{align}
\Vert y_{xx}\Vert_{L^{\infty}(L^2)}^2 + (\lambda C_0^\beta+\ve)\Vert y_{xxxx}\Vert_{L^2(L^2)}^2 \leq 
C(\Vert u_x\Vert_{L^2(L^2)}^2 + 1)\, . \label{eq:thin_film_eps_apriori_h4}
\end{align}
\end{lemma}
}
\begin{proof}
We rewrite the main part of the equation \eqref{eq:thin_film_eps} in non-divergence form,
\begin{equation}\label{rtm1}\big( [\lambda\vert y\vert^\beta +\ve]y_{xxx}\big)_x = \beta \lambda y^{\beta-1} y_x y_{xxx} + [\lambda\vert y\vert^\beta+\ve]y_{xxxx}. 
\end{equation}
Multiply \eqref{rtm1}  with $y_{xxxx}$, integrate over $\Omega$ to arrive for $\sigma>0$ at
\begin{align}
\frac12\frac{d}{dt} \Vert y_{xx}\Vert^2 + \underbrace{\int_{\Omega} [\lambda\vert y\vert^\beta+\ve]y_{xxxx}^2 \dd x }_{=:I_1} \leq{}& \underbrace{-\int_{\Omega} \beta \lambda y^{\beta -1}y_x y_{xxx}y_{xxxx} \dd x}_{=:I_2} \label{eq:thin_film_eps_apriori_h4_zwischenschritt} \\
&+ \sigma\Vert y_{xxxx}\Vert^2 + C(\sigma) \Vert u_x\Vert^2. \nonumber
\end{align}
{ We calculate for $\sigma>0$, using $H^1(\Omega) \subset L^{\infty}(\Omega)$,
\begin{align*}
I_1 &\geq (\lambda C_0^\beta +\ve)\Vert y_{xxxx}\Vert^2_{L^2}, \\
I_2 &\leq C(\sigma) \Vert y\Vert_{L^{\infty}}^{2(\beta-1)} \bigl(\Vert y_x\Vert^2_{L^{2}} + \Vert y_{xx}\Vert^2_{L^2}\bigr) \Vert y_{xxx}\Vert^2_{L^2} + \sigma \Vert y_{xxxx}\Vert^2_{L^2}\, .
\end{align*}

We now absorb the terms on the right-hand side which are lead by $\sigma>0$ in \eqref{eq:thin_film_eps_apriori_h4_zwischenschritt} into the lower bound of $I_1$. Those lead by $C(\sigma)$ are integrable in time by \eqref{eq:thin_film_eps_apriori_weak}, and we deduce \eqref{eq:thin_film_eps_apriori_h4} with the help of Gronwall's lemma.
}
\end{proof}

\subsection{Existence}

For every $\ve>0$, the regularized equation \eqref{eq:thin_film_eps} has a unique weak solution.

\begin{lemma} \label{lemma:thin_film_eps_existence}
Let $\lambda,\ve>0$, and $u\in L^2(H^1_0)$. Then \eqref{eq:thin_film_eps} has a {unique} weak solution 
{$y\in L^2(H^4)\cap H^1(L^2)$}.
\end{lemma}
\begin{proof}
This follows from standard parabolic theory since the leading part of the equation is uniformly parabolic.
\end{proof}

In contrast, it is not clear if there is a solution of \eqref{eq:thin_film}  for general $u$ and, even more, whether it is non-negative. The following lemma asserts non-negativity of solutions of \eqref{eq:thin_film} for $u = 0$; the proof of it was first given in~\cite[Theorems 3.1 and 4.1]{bernis_friedman}, and uses
the entropy functional $H[y] = \frac{1}{(\beta-1)(\beta-2)} \int_{\Omega} y^{2-\beta}\, {\rm d}x$. 
This result for the PDE \eqref{eq:thin_film} is relevant to later infer solvability for  Problem~\ref{prob:thin_film_opti}, hence we
here recapitulate its proof from \cite{bernis_friedman}.

 \begin{lemma} \label{lemma:thin_film_g0_gleich_0_positivity}
Let $\lambda>0$, $\beta \geq 4$, and {$E[y_0] + H[y_0] < \infty$.} There exist a constant
$\widetilde{C}_0 = \widetilde{C}_0(E[y_0], H[y_0], \beta, \Omega) >0$ and  
$u\in L^2(H_0^1)$ such that the unique weak solution
$y:\Omega_T \rightarrow {\mathbb R}$ of (\ref{eq:thin_film}) satisfies $y\geq 2C_0$ in $\overline{\Omega}_T$ for all $0 < C_0 \leq \widetilde{C}_0$.
\end{lemma}
 
\begin{proof} 
Consider (\ref{eq:thin_film}) with $u = 0$. By assumption, $H[y_0] < \infty$, and thus $y_0>0$. By
{Lemmas~\ref{lemma:thin_film_hoelder_space} and \ref{lemma:thin_film_hoelder_time}}, the solution $y:[0,T] \times \overline{\Omega} \rightarrow {\mathbb R}$ is continuous. We argue by contraduction, and may thus infer existence of
$\widehat{t} >0$ and $\widehat{x} \in \overline{\Omega}$ such that $0 = \min_{x \in \overline{\Omega}}  y(\widehat{t},x) = \min y(\widehat{t},\widehat{x})$ otherwise.  However, along a solution, equation \eqref{eq:thin_film} with $u = 0$, the chain rule, and integration by parts lead to
\begin{equation}\label{entropie_1} - \frac{{\rm d}}{{\rm d}t} H[y(t)] = - \frac{1}{1-\beta} \int_{\Omega} y^{1-\beta} y_t\, {\rm d}x = - \int_{\Omega}
y^{-\beta} y_x y^\beta y_{xxx}\, {\rm d}x = \int_{\Omega} y^2_{xx}\, {\rm d}x \geq 0\, . 
\end{equation}
By Lemma~\ref{lemma:thin_film_hoelder_space}  we have  that $0 \leq y(\widehat{t},x) \leq H_{\rm space} \vert x - \widehat{x}\vert^{1/2}$ for every $x \in \overline{\Omega}$. As a consequence,
we may resume from (\ref{entropie_1}) that 
$$ \int_{\Omega} \vert x - \widehat{x}\vert^{-(\frac{\beta}{2} - 1)}\, {\rm d}x \leq H_{\rm space}^{\beta-2}  \int_{\Omega} y(\widehat{t}, x)^{2-\beta}\, {\rm d}x \leq K H[y_0] < \infty\, ,$$
where $K := H_{\rm space}^{\beta-2} (\beta-1)(\beta-2)$.
But since $\beta \geq 4$, the leading integral  is infinite by assumption, which is a contradiction. 
\end{proof}

There are two ways to construct a solution of \eqref{eq:thin_film} by a sequence $\{y_{\ve}\}$ solving \eqref{eq:thin_film_eps} for a sequence $\ve\to0$: one strategy is to restrict to more regular right-hand sides $u\in L^2(H^2)$ which allow uniform estimates as in Lemma~\ref{lemma:thin_film_apriori_h4} with respect to $\ve>0$. Another strategy which we will use here is the following: if all iterates $y_{\ve}$ have a pointwise lower bound which is uniformly bounded away from zero with respect to $\ve>0$, then it is also possible to pass to the limit, even without the use of more regular right-hand sides $u$. 
The following two lemmas reflect both situations independently.

\begin{lemma} \label{lemma:thin_film_eps_convergence_uniform}
Let $\lambda>0$, $u\in L^2(H_0^1\cap H^2)$, and $\lbrace y_{\ve}\rbrace$ be a sequence of solutions of \eqref{eq:thin_film_eps}.  There exists $y\in H^1(L^2)\cap L^2(H^4)$ 
such that $y_{\ve}\to y$ uniformly in $\Omega_T$ for $\ve\to 0$, and y solves
\eqref{eq:thin_film}.
\end{lemma}
\begin{proof}
We use integration by parts for the term 
$-(u_x,y_{xx})=(u_{xx},y_x)\leq \Vert u_{xx}\Vert^2 + \Vert y_x\Vert^2$ in \eqref{eq:thin_film_hoelder_space_nach_testen}.
As a consequence, we obtain uniform bounds for $y_{\ve}\in \Con(\Con^{0,\frac12})$ with respect to~$\varepsilon >0$ as stated
in Lemma~\ref{lemma:thin_film_hoelder_space} without the requirement on lower bounds.
Because of Lemma \ref{lemma:thin_film_hoelder_time}, the sequence $\{y_{\ve}\}$ is uniformly bounded in $\Con^{0,\frac18}(\Con^{0,\frac12})$, such that
there exist a subsequence and $y \in \Con(\Con^{0,\frac12})$ such that $y_{\ve}\to y$ uniformly in $\Omega_T$, which solves
\eqref{eq:thin_film}. By \cite[Theorems 3.1 and 4.1]{bernis_friedman}, solutions of \eqref{eq:thin_film} are unique, such that the limit of every convergent subsequence of $\{ y_{\varepsilon}\}$ must be  a solution of of \eqref{eq:thin_film}. Hence the whole sequence $\{ y_\varepsilon\}$ converges to the solution $y$.
\end{proof}

\begin{lemma} \label{lemma:thin_film_eps_convergence_weak}
Let $\lambda>0$, $u\in L^2(H_0^1)$, and  $\lbrace y_{\ve}\rbrace$ be a sequence of  solutions $y_{\ve}$ of \eqref{eq:thin_film_eps}. Assume $y_{\ve}\geq C_0>0$ for all $\ve>0$. There exists $y\in H^1(L^2)\cap L^2(H^4)$ such that $y_{\ve}\to y$ uniformly in $\Omega_T$ for $\ve\to 0$, and $y$ solves \eqref{eq:thin_film}.
\end{lemma}
\begin{proof}
Because of $y_{\ve}\geq C_0$ uniformly in $\ve>0$, the estimate \eqref{eq:thin_film_hoelder_space_nach_testen} 
yields a uniform bound for $y_{\ve}$ in $L^2(H^3)\cap H^1((H^1)^*)$.
Moreover, the proof of Lemma \ref{lemma:thin_film_apriori_h4} shows that $\{y_{\ve}\}$ is uniformly bounded in $L^2(H^4)\cap H^1(L^2)$. Hence, there exist a limiting function $y\in L^2(H^4)\cap H^1(L^2)$, and a subsequence (not relabelled) such that $y_{\ve}\rightharpoonup y$ weakly in $L^2(H^4)\cap H^1(L^2)$. 
By the embedding $L^2(H^4)\cap H^1(L^2)\subset \Con(\Omega_T)$ and Lemmas~\ref{lemma:thin_film_hoelder_space} and \ref{lemma:thin_film_hoelder_time}, we have that $\lbrace y_{\ve}\rbrace$ is equicontinuous, uniformly bounded, hence for another subsequence (not relabelled) $y_{\ve}\to y$ uniformly in $\Omega_T$.

In order to show that $y$ solves \eqref{eq:thin_film}, we perform limits term by term in \eqref{eq:thin_film_eps};
it is due to the uniform lower bound for each of the $\{y_{\ve}\}$ that this is more easy than in the proof of Lemma~\ref{lemma:thin_film_eps_convergence_uniform}. It is clear for the linear terms; for the nonlinear term, we calculate for $\varphi\in\Con^{\infty}(\Omega_T)$ along the given subsequence 
\[ \big(\lambda \vert y_{\ve}\vert^\beta y_{\ve,xxx} - \lambda \vert y\vert^\beta y_{xxx},\varphi_x\big) = \big( [\lambda \vert y_{\ve}\vert^\beta -\lambda\vert y\vert^\beta ]y_{\ve,xxx},\varphi_x\big) + \big( \lambda\vert y\vert^\beta [y_{\ve,xxx}-y_{xxx}],\varphi_x\big) \to 0. \]
Because of the uniqueness of solutions of \eqref{eq:thin_film}, we may again infer convergence of the whole sequence $\{ y_{\varepsilon} \}$.
\end{proof}

The assumption $y_{\ve}\geq C_0>0$ for all $\{ y_{\ve} \}$ in Lemma~\ref{lemma:thin_film_eps_convergence_weak} is rather strong. However, such a sequence exists 
for an optimal control problem which involves \eqref{eq:thin_film_eps} together with $y_{\ve}\geq C_0$; cf.~Problem~\ref{prob:thin_film_opti_thin_film_eps} below.

\section{Analysis of the optimization problem without regularization} \label{sec:thin_film_opt_original}

We show solvability for the original optimization Problem \ref{prob:thin_film_opti} and derive necessary optimality conditions. 
The constant $C_0>0$ has to be chosen in such a way that Lemma~\ref{lemma:thin_film_g0_gleich_0_positivity} holds.

\begin{thm} \label{thm:thin_film_existence_opti_mod}
{Let $\beta \geq 4$ and $0 < C_0 \leq \widetilde{C}_0$, where $\widetilde{C}_0 = \widetilde{C}_0(E[y_0], H[y_0], \Omega, \beta)>0$.} Then Problem~\ref{prob:thin_film_opti} has at least one solution.
\end{thm}
\begin{proof}
{\em Step 1:}
By Lemma \ref{lemma:thin_film_g0_gleich_0_positivity}, there exists at least one $\underline{u}\in L^2(H_0^1)$ such that all side constraints (i.e., the equation \eqref{eq:thin_film}, and $y = y(\underline{u})\geq C_0$ in $\Omega_T$) are satisfied. Therefore, we have 
\[ \inf J(y,u)=:J^*>-\infty, \]
where the infimum is taken over all feasible pairs $(y,u)$.

{\em Step 2:}
Hence, there exists a sequence $\lbrace (y_i,u_i) \rbrace$ fulfilling \eqref{eq:thin_film}, $y_i\geq C_0$, such that $J(y_i,u_i)\searrow J^*$. By definition of the functional $J$, $u_i$ is bounded in $L^2(H_0^1)$ and there exists a $u\in L^2(H_0^1)$ such that $u_i\rightharpoonup u$ weakly in $L^2(H_0^1)$ (up to subsequences).

By Lemma \ref{lemma:thin_film_apriori_h4}, members of the sequence $\{y_i\} \subset L^2(H^4)\cap H^1(L^2) $ are bounded by corresponding ones in $\{u_i\} \subset L^2(H^1_0)$, hence $y_i$ is uniformly bounded in $L^2(H^4)\cap H^1(L^2)$. By Lemma \ref{lemma:thin_film_eps_convergence_weak}, there exists a $y\in L^2(H^4)\cap H^1(L^2)$ such that $y_i\rightharpoonup y$ weakly in $L^2(H^4)\cap H^1(L^2)$ and $y_i\to y$ uniformly in $\Omega_T$, and $y$ solves \eqref{eq:thin_film}. Moreover, we have $y\geq C_0$.

{\em Step 3:} By the weak lower semicontinuity of the functional $J$, $(y,u)$ is a minimizer of Problem~\ref{prob:thin_film_opti}. 
\end{proof}

In the rest of this section, we derive necessary optimality conditions for a minimum obtained by Theorem~\ref{thm:thin_film_existence_opti_mod}. The key step to derive this is the following abstract result about optimal control problems with state constraints, which is obtained in \cite{alibert_raymond}.

\begin{lemma} \label{lemma:thin_film_alibert_raymond}
Let $X,V,W$ be Banach spaces, $U$ be a separable Banach space, let $J:X\times U\to \R$, $G:X\times U\to V$, $H:X\to W$ be mappings, and $C\subseteq W$ be a set.

Let $(\bar x,\bar u)\in X\times U$ be a minimum of the optimal control problem 
\[ J(\bar x,\bar u) = \min_{(x,u)\in S} J(x,u)\]
with 
\[ S:=\big\lbrace (x,u)\in X\times U: \; G(x,u)=0, \; H(x)\in C\big\rbrace \]
and let the following assumptions be true:
\begin{enumerate}
\item $G:X\times U\to V$ is Frechet differentiable at $(\bar x,\bar u)$,
\item $H:X\to W$ is Frechet differentiable at $\bar x$,
\item $\varnothing\neq C\subseteq W$ is a convex subset with nonempty interior (measured in the topology of $W$),
\item $G_x'(\bar x,\bar u):X\to V$ is surjective.
\end{enumerate}
Then there exist $(p,\mu,\zeta)\in V^*\times W^*\times \R$ such that
\begin{subequations}
\begin{align}
\zeta \langle J_x'(\bar x,\bar u),x\rangle_{X,X^*} + \langle p,G_x'(\bar x,\bar u)x\rangle_{V,V^*} + \langle \mu,H'(\bar x)x\rangle_{W,W^*} &= 0 \quad \forall x\in X, \\
\zeta \langle J_u'(\bar x,\bar u),u\rangle_{U,U^*} + \langle p,G_u'(\bar x,\bar u)u\rangle_{V,V^*} &= 0 \quad \forall u\in U, \\
\zeta&\geq 0, \\
\langle \mu,w-H(\bar x)\rangle_{W,W^*} &\leq 0 \quad \forall w\in C
\end{align} \label{eq:thin_film_alibert_raymond_nec_opti}%
\end{subequations}
and if $\zeta=0$ then $\langle \mu, w\rangle_{W,W^*}\neq 0$ for some $w\in C$.

If we additionally assume that there exists $(\underline{x},\underline{u})\in X\times U$ such that
\begin{subequations}
\begin{align}
G_x'(\bar x,\bar u)\underline{x} + G_u'(\bar x,\bar u)(\underline{u}-\bar u)&=0, \label{eq:thin_film_alibert_raymond_lagrange_additional_1} \\
H(\bar x)+H'(\bar x)\underline{x}&\in \interior C, \label{eq:thin_film_alibert_raymond_lagrange_additional_2}
\end{align}
\end{subequations}
then we can take $\zeta=1$.
\end{lemma}

We now apply this general result to our setup in Problem~\ref{prob:thin_film_opti}. Let $X:=W:= L^2(H^4)\cap H^1(L^2)$,  $U:=L^2(H_0^1)$, $V:=L^2(L^2)$, and $C:=\lbrace v\in W: \; v\geq C_0 \text{ in } \Omega_T\rbrace$. Since $W\subset \Con(\overline{\Omega_T})$ by Sobolev embeddings, the set $C$ is well-defined.

The function $G$ is given by 
\[ G(y,u) := y_t + \big( \lambda\vert y\vert^{\beta}y_{xxx}\big)_x - u_x, \]
while $H$ is given by $H(y):=y$. We omit initial conditions and boundary conditions in $G$, which may be treated by standard methods; see, e.g.,~\cite[Section 2.6]{gunzburger_flow_control_blau}.

\begin{lemma} \label{lemma:thin_film_well-def_G_H}%
\leavevmode
\begin{enumerate}
\item The function $G:X\times U\to V$ is well-defined.
\item The function $H:X\to W$ is well-defined.
\item The set $C$ is convex with nonempty interior (measured in the topology of $W$).
\end{enumerate}
\end{lemma}
\begin{proof}
\begin{enumerate}
\item This follows from Lemma~\ref{lemma:thin_film_apriori_h4}.
\item Clear by definition.
\item Clearly, the set $C$ is convex, since it is the intersection of two convex sets. We note that the set $\tilde{C}:= \lbrace v\in \Con(\overline{\Omega_T}): \; v\geq C_0 \text{ in } \Omega_T\rbrace$ has nonempty interior (e.g., $\hat{v} =  2C_0$ is an interior point), i.e., there exist a point $\hat{v}\in C$ and $r>0$ such that $B_r(\hat{v})\subset \tilde{C}$. Without loss of generality, we can assume that $\hat{v}\in W$ due to the density of $W\subset \Con(\overline{\Omega}_T)$. Since the embedding $\id:W\to \Con(\overline{\Omega_T})$ is continuous by Sobolev embeddings, the preimage $\id^{-1}(B_r(\hat{v}))\subset C$ is open, hence there exists an open neighborhood of $\id^{-1}(\hat{v})$, and this $C$ has nonempty interior in the topology of $W$.
\end{enumerate}
\end{proof}

We now check that the remaining assumptions in Lemma~\ref{lemma:thin_film_alibert_raymond} are valid. In order to write down \eqref{eq:thin_film_alibert_raymond_nec_opti}, we have to show that $G_x'(\bar x,\bar u):X\to V$ is surjective, which is done in the following.\

\begin{lemma} \label{lemma:thin_film_frechet_ableitung_G} 
The function $G$ as defined above has the following Frechet derivatives.
\begin{align*}
\left\langle G_{y}'\big( \bar y,\bar u\big),\delta y\right\rangle ={}& (\delta y)_t + \big(\langle \beta\lambda {\bar y}^{\beta-1},\delta y\rangle \bar y_{xxx}\big)_x + \big(\lambda \vert \bar y\vert^\beta (\delta y)_{xxx}\big)_x && \forall \delta y\in X, \\
\left\langle G_{u}'\big( \bar y,\bar u\big),\delta u\right\rangle ={}& -(\delta u)_x && \forall {\delta u}\in U.
\end{align*}
\end{lemma}
\begin{proof}
The function $G$ is smooth and the derivation of it is a straight forward calculation.
\end{proof}

\begin{lemma} \label{lemma:thin_film_ableitung_surj}
For every $\Phi\in L^2(L^2)$, there exists  $v\in L^2(H^4)\cap H^1(L^2)$ such that
\begin{align} 
\left\langle G_{y}'\big(\bar y,\bar u\big),v\right\rangle = \Phi \label{eq:thin_film_G_surj_lin_thin_film_to_show}
\end{align}
together with the initial conditions $v(0,.)=0$ as well as the boundary conditions $v_x=v_{xxx}=0$ in $a,b$.
\end{lemma}
\begin{proof}
Inserting the derivative of $G$ with respect to $y$ by Lemma~\ref{lemma:thin_film_frechet_ableitung_G} in equation \eqref{eq:thin_film_G_surj_lin_thin_film_to_show} leads to
\begin{align} 
v_t + \big(\lambda\vert \bar y\vert^\beta v_{xxx}\big)_x + \text{ lower order terms } = \Phi. \label{eq:thin_film_surjective_derivative_to_solve}
\end{align}
For a test function $\varphi\in X$, we write
\begin{align} 
\left\langle \big(\lambda\vert \bar y\vert^\beta v_{xxx}\big)_x,\varphi\right\rangle = -\left\langle \lambda\vert \bar y\vert^\beta v_{xxx},\varphi_x\right\rangle = \left\langle \lambda \vert \bar y\vert^\beta v_{xx},\varphi_{xx}\right\rangle + \left\langle \beta\lambda (\bar y)^{\beta-1}\bar y_x v_{xx},\varphi_x\right\rangle. \label{eq:thin_film_surjective_derivative_nondivergenceform}
\end{align}
Since $\beta\lambda (\bar y)^{\beta-1}\geq \beta\lambda C_0^{\beta-1}>0$, we can estimate the last term in \eqref{eq:thin_film_surjective_derivative_nondivergenceform} as follows
\[ \left\langle \beta \lambda (\bar y)^{\beta-1}\bar y_{x}v_{xx},\varphi_x\right\rangle \leq \sigma \Vert \lambda \vert \bar y\vert^{\beta-2} v_{xx}\Vert^{2} + C(\sigma)\Vert \bar y \bar y_x \varphi_x\Vert^2 \]
with $\sigma>0$. The remaining term in \eqref{eq:thin_film_surjective_derivative_nondivergenceform} is either uniformly $H^2$-coercive (since $\bar y\geq C_0$) or is of lower order. Therefore, there exists a solution $v\in L^2(H^2)\cap H^1(H^{-1})$ of \eqref{eq:thin_film_surjective_derivative_to_solve}.

As in the proof of Lemma \ref{lemma:thin_film_apriori_h4}, we can write the operator in non-divergence form,
\begin{align} 
\big(\lambda \vert \bar y\vert^{\beta} v_{xxx}\big)_x = \lambda\vert \bar y\vert^{\beta} v_{xxxx} + \beta \lambda(\bar y)^{\beta-1}\bar y_x v_{xxx}, \label{eq:thin_film_proof_surj_rewrite_leading_part}
\end{align}
i.e., the leading part of the equation \eqref{eq:thin_film_surjective_derivative_to_solve} is uniformly elliptic since $\bar y\geq C_0$. We proceed similarly to the proof in Lemma \ref{lemma:thin_film_apriori_h4} and multiply the equation with $v_{xxxx}$ to absorb the lower order terms into the leading term in \eqref{eq:thin_film_proof_surj_rewrite_leading_part}. We may then infer that the solution $v$ is as regular as claimed.
\end{proof}

We now show that the regular point conditions \eqref{eq:thin_film_alibert_raymond_lagrange_additional_1} and \eqref{eq:thin_film_alibert_raymond_lagrange_additional_2} from Lemma \ref{lemma:thin_film_alibert_raymond} are fulfilled. For this goal, it is important to make use of the surjectivity of the derivative of $G$.

\begin{lemma} \label{lemma:thin_film_gueltigkeit_alibert_raymond_additional}
There exists $(\underline{y},\underline{u})\in X\times U$ such that \eqref{eq:thin_film_alibert_raymond_lagrange_additional_1} and \eqref{eq:thin_film_alibert_raymond_lagrange_additional_2} are fulfilled.
\end{lemma}
\begin{proof}
{\em Step 1:} First  note that $\interior C-H(\bar y)=\lbrace f\in\Con(\Omega_T): \; f> C_0-\bar y\rbrace$. Since $H'(\bar y)\underline{y}=\underline{y}$, we have to choose $\underline{y}\in X$ such that $\underline{y}>C_0-\bar y$ in $\Omega_T$ to meet \eqref{eq:thin_film_alibert_raymond_lagrange_additional_2}, which is always possible (e.g., choose $\underline{y}=2C_0$).

{\em Step 2:} Consider the first component of the equation \eqref{eq:thin_film_alibert_raymond_lagrange_additional_1}
\[ G_y'(\bar y,\bar u)\underline{y} + G_u'(\bar y,\bar u)(\underline{u}-\bar u)=0, \]
{which can be written as}
\begin{align}
\left\langle G_{y}'\big(\bar y,\bar u\big),\underline{y}\right\rangle = \underline{u}_x - \bar{u}_x=:\tilde{u}_x \label{eq:thin_film_alibert_raymond_additional_proof1}
\end{align}
due to Lemma \ref{lemma:thin_film_frechet_ableitung_G}. By Lemma \ref{lemma:thin_film_ableitung_surj}, the left-hand side of \eqref{eq:thin_film_alibert_raymond_additional_proof1} is surjective, i.e., there exists a $\tilde{u}_x\in L^2(L^2)$ such that \eqref{eq:thin_film_alibert_raymond_additional_proof1} holds. Since $\bar{u}_x$ is known and we do not have additional constraints on $u$, there exists a $\underline{u}_x\in L^2(L^2)$ such that \eqref{eq:thin_film_alibert_raymond_additional_proof1} holds.
To summarize, we have constructed $(\underline{y},\underline{u})\in X\times U$ such that both conditions \eqref{eq:thin_film_alibert_raymond_lagrange_additional_1} and \eqref{eq:thin_film_alibert_raymond_lagrange_additional_2} hold.
\end{proof}

\begin{thm} \label{thm:thin_film_opt_cond_original}
Let $(y,u)$ be a solution of Problem \ref{prob:thin_film_opti}. Then, there exist $z\in L^2(L^2)$, and $\mu\in(L^2(H^4)\cap H^1(L^2))^*$ such that the following optimality conditions are fulfilled,
\begin{subequations}
\begin{align}
y_t &= -\big(\lambda \vert y\vert^\beta y_{xxx}\big)_x + u_x, \label{eq:thin_film_opt_cond_state} \\
y &\geq C_0 \label{eq:thin_film_opt_cond_state_constraint} \\
0&\geq \langle w-y,\mu\rangle \quad \forall X\ni w\geq C_0, \label{eq:thin_film_opt_cond_compl_cond} \\
0 &= \langle y-\tilde{y},\varphi\rangle + \big\langle z,\varphi_t + \beta \lambda y^{\beta-1} y_{xxx}\varphi)_x\big\rangle + \big\langle z,\big( \lambda\vert y\vert^\beta \varphi_{xxx}\big)_x\big\rangle + \langle \varphi,\mu\rangle \quad \forall \varphi\in X, \label{eq:thin_film_opt_cond_adjoint1} \\
0 &= -\alpha u_{xx} + z_x \label{eq:thin_film_opt_cond_opti_cond_u}
\end{align} \label{eq:thin_film_opt_cond}%
\end{subequations}
together with initial conditions $y(0,.)=y_0$, $z(T,.)=0$, and boundary conditions $y_x=y_{xxx}=z_x=z_{xxx}=0$ in $a,b$.
\end{thm}
\begin{proof}
We use Lemma \ref{lemma:thin_film_alibert_raymond} whose hypotheses are fulfilled by Lemmas \ref{lemma:thin_film_well-def_G_H}, \ref{lemma:thin_film_frechet_ableitung_G}, \ref{lemma:thin_film_ableitung_surj}, and \ref{lemma:thin_film_gueltigkeit_alibert_raymond_additional}.
\end{proof}

\section{Optimization with regularization in the equation} \label{sec:thin_film_opti_reg}

We consider a modification of Problem~\ref{prob:thin_film_opti} where the state equation is regularized. We show solvability and derive corresponding optimality conditions in Theorems~\ref{thm:thin_film_existence_opti_regularization} and \ref{thm:thin_film_opti_cond_eps}.  Theorem~\ref{thm:thin_film_convergence_eps_to_zero} validates that
solutions of this problem converge to those of \eqref{eq:thin_film_opt_cond}. 

\begin{prob} \label{prob:thin_film_opti_thin_film_eps}
Let $\lambda,\alpha,\ve>0$, $\tilde{y}\in L^2(\Omega_T)$. Minimize $J:L^2(H^4)\cap H^1(L^2)\times L^2(H_0^1)\to\R$
\[ J(y,u):= \frac12 \int_0^T \int_{\Omega} \vert y-\tilde{y}\vert^2 \dd x \dd t + \frac{\alpha}2\int_0^T \int_{\Omega} \vert u\vert^2 \dd x \dd t \]
subject to $y\geq C_0>0$ and \eqref{eq:thin_film_eps} together with initial condition $y(0,.)=y_0$, boundary conditions $y_x=y_{xxx}=0$ in $a,b$.
\end{prob}

\begin{thm} \label{thm:thin_film_existence_opti_regularization}
{Suppose $\beta \geq 4$, $0 < C_0 \leq \widetilde{C}_0$, and $0 < \varepsilon \leq \widetilde{\varepsilon}_0$
for $\widetilde{\varepsilon}_0 = \widetilde{\varepsilon}_0({C}_0) > 0$. Then Problem \ref{prob:thin_film_opti_thin_film_eps} has at least one solution.}
\end{thm}
{ \begin{proof}
Let $y_\varepsilon:\Omega_T \rightarrow {\mathbb R}$ be the weak solution of \eqref{eq:thin_film_eps} with $u=0$. By Lemma~\ref{lemma:thin_film_eps_convergence_uniform} there exists $\widetilde{\varepsilon}_0 >0$  such that
$\sup_{0 < \varepsilon \leq \widetilde{\varepsilon}_0}\Vert y - y_{\varepsilon}\Vert_{\Con(\Omega_T)} \leq C_0$, where $y \geq 2C_0$ in $\Omega_T$ solves \eqref{eq:thin_film} with $u=0$; see Lemma~\ref{lemma:thin_film_g0_gleich_0_positivity}.
The remaining two steps in the proof of Theorem~\ref{thm:thin_film_existence_opti_mod} are now applicable.
\end{proof}
}

We may use Lemma \ref{lemma:thin_film_alibert_raymond} to derive necessary optimality conditions since all requirements are fullfilled as in the proof of Theorem \ref{thm:thin_film_opt_cond_original}.
\begin{thm} \label{thm:thin_film_opti_cond_eps}
Let $(y,u)$ be a solution of Problem \ref{prob:thin_film_opti_thin_film_eps}. Then, there exist Lagrange multipliers $z\in L^2(L^2)$, and $\mu\in (L^2(H^4)\cap H^1(L^2))^*$ such that the following equations are fulfilled,
\begin{subequations}
\begin{align}
y_t &= -\big([\lambda\vert y\vert^\beta+\ve]y_{xxx}\big)_x + u_x, \label{eq:thin_film_opt_cond_eps_state} \\
y &\geq C_0 \label{eq:thin_film_opt_cond_state_eps_constraint} \\
0&\geq \langle w-y,\mu\rangle \quad \forall X\ni w\geq C_0, \label{eq:thin_film_opt_cond_eps_compl_cond} \\
0 &= \langle y-\tilde{y},\varphi\rangle + \big\langle z,\varphi_t + (\beta\lambda y^{\beta-1}y_{xxx}\varphi)_x\big\rangle + \big\langle z,\big( [\lambda\vert y^\beta+\ve]\varphi_{xxx}\big)_x\big\rangle + \langle \varphi,\mu\rangle \quad \forall \varphi\in X, \label{eq:thin_film_opt_cond_eps_adjoint1} \\
0 &= -\alpha u_{xx} + z_x \label{eq:thin_film_opt_cond_eps_opti_cond_u}
\end{align} \label{eq:thin_film_opt_cond_eps}%
\end{subequations}
together with initial conditions $y(0,.)=y_0$, $z(T,.)=0$; boundary conditions $y_x=y_{xxx}=z_x=z_{xxx}=0$ in $a,b$.
\end{thm}

We are now able to state and prove the first main theorem of the paper.

\begin{thm} \label{thm:thin_film_convergence_eps_to_zero}
Let $\lbrace (y_{\ve},u_{\ve})\rbrace$ be a sequence of solutions of Problem \ref{prob:thin_film_opti_thin_film_eps}, and $\lbrace z_{\ve},\mu_{\ve}\rbrace$ corresponding Lagrange multipliers from Theorem~\ref{thm:thin_film_opti_cond_eps}. Then, there exist $(y^*,u^*)\in \big(H^1(L^2)\cap L^2(H^4)\big)\times L^2(H_0^1)$ and $(z^*,\mu^*)\in L^2(L^2)\times (L^2(H^4)\cap H^1(L^2))^*$ such that $(y_{\ve},u_{\ve})\rightharpoonup (y^*,u^*)$ weakly in $\big(H^1(L^2)\cap L^2(H^4)\big)\times L^2(H_0^1)$ and $(z_{\ve},\mu_{\ve})\rightharpoonup (z^*,\mu^*)$ weakly in $L^2(L^2)\times (L^2(H^4)\cap H^1(L^2))^*$ for $\ve\to 0$ (up to a subsequence, respectively). The limiting functions $(y^*,u^*,z^*,\mu^*)$ are a solution of \eqref{eq:thin_film_opt_cond}.
\end{thm}
\begin{proof}
{\em Step 1:} We first show that $\{u_{\ve}\}$ is uniformly bounded in $L^2(H_0^1)$: to do so, we give a function $\bar u$ and a corresponding solution $\bar y_{\ve}$, which is feasible for every $\ve>0$ small enough, i.e., which is solving \eqref{eq:thin_film_eps} together with $\bar y_{\ve} \geq C_0$. For $u = 0$ and $\ve=0$, by Lemma \ref{lemma:thin_film_g0_gleich_0_positivity} there exists a weak solution $\bar y$ of \eqref{eq:thin_film} such that  $\bar y\geq 2C_0$. Let $\lbrace y_{\ve}^{(0)}\rbrace$ be the sequence of solutions of \eqref{eq:thin_film_eps} where $u = 0$. Then there exists $y:\Omega_T\to \R$ such that $\bar y_{\ve}^{(0)}\to y$ uniformly for $\ve\to 0$; cf.~Lemma \ref{lemma:thin_film_eps_convergence_uniform}. Hence there exists an $\ve_0>0$ such that $\bar y_{\ve}^{(0)}\geq C_0$ for every $0<\ve\leq \ve_0$. 

Since $\lbrace \bar y_{\ve}\rbrace$ is uniformly bounded (with respect to $\ve>0$) in $L^2(H^4)\cap H^1(L^2)$ by a constant depending on the fixed norm of $\bar u = 0$, we may deduce that the solution $(y_{\ve},u_{\ve})$ of Problem \ref{prob:thin_film_opti_thin_film_eps} satisfies $J(y_{\ve},u_{\ve})\leq J(y_{\ve}^{(0)},0)<\infty$, i.e., by construction of the functional $J$, the sequence $\{u_{\ve}\}$ is bounded uniformly in $L^2(H_0^1)$. Hence there exists a $u^*\in L^2(H_0^1)$ such that $u_{\ve}\rightharpoonup u^*$ weakly in $L^2(H_0^1)$.

{\em Step 2:} By Lemma \ref{lemma:thin_film_apriori_h4}, the solution $y_{\ve}$ of \eqref{eq:thin_film_eps} is uniformly bounded (with respect to $\ve>0$) in $L^2(H^4)\cap H^1(L^2)$, i.e., there exists $y^*\in L^2(H^4)\cap H^1(L^2)$ such that $y_{\ve}\rightharpoonup y^*$ weakly in $L^2(H^4)\cap H^1(L^2)$. Since all $y_{\ve}\geq C_0$, we have $y^*\geq C_0$ and $y^*$ solves \eqref{eq:thin_film} by Lemma \ref{lemma:thin_film_eps_convergence_weak}.

{\em Step 3:} We show that the Lagrange multipliers $\{(z_{\ve},\mu_{\ve})\}$ are uniformly bounded (with respect to $\ve$) in their corresponding spaces: Since $\{u_{\ve}\}$ is bounded in $L^2(H_0^1)$, we have $\{z_{\ve,x}\}$ bounded uniformly in $L^2(H^{-1})$. We will now consider \eqref{eq:thin_film_opt_cond_eps_adjoint1} and may show that $\mu_{\ve}$ is uniformly bounded in $(L^2(H^4)\cap H^1(L^2))^*$, i.e., we have to show that
\[ \Vert \mu_{\ve}\Vert_{(L^2(H^4)\cap H^1(L^2))^*} = \sup_{\stackrel{\psi\in L^2(H^4)\cap H^1(L^2)}{\Vert \psi\Vert_{L^2(H^4)\cap H^1(L^2)}\leq 1}} \vert \langle \mu_{\ve},\psi\rangle \vert \]
is bounded independently from $\ve>0$. Since $\mu_{\ve}$ is on the right-hand side of \eqref{eq:thin_film_opt_cond_eps_adjoint1}, we can represent $\mu_{\ve}$ by means of $y_{\ve}$ and $z_{\ve}$, i.e., we have
\begin{align*}
\vert \langle \mu_{\ve},\psi\rangle\vert \leq{}& \vert \langle y_{\ve}-\tilde{y},\psi\rangle\vert + \vert \langle z_{\ve},\psi_t\rangle\vert + \vert \langle z_{\ve,x},\beta\lambda y_{\ve}^{\beta-1}y_{\ve,xxx}\psi\rangle\vert \\
&+ \vert \langle z_{\ve,x},[\lambda \vert y_{\ve}\vert^{\beta}+\ve]\psi_{xxx}\rangle\vert =: I_1 + I_2 + I_3 + I_4.
\end{align*}
We estimate those terms as follows and use the bounds from the first steps (and Sobolev embeddings),\begin{align*}
I_1 &\leq \Vert y_{\ve}-\tilde{y}\Vert\Vert \psi\Vert \leq C, \\
 I_2 & \leq \Vert z_{\ve}\Vert \Vert \psi_t\Vert \leq C, \\
I_3 &\leq \Vert z_{\ve} \Vert_{{ L^2(L^{2})}} \Vert \beta \lambda y_{\ve}^{\beta-1}y_{\ve,xxx}\psi\Vert_{{ L^2(H^{1})}} \leq C, \\
 I_4 &\leq \Vert z_{\ve,x}\Vert_{{ L^2(H^{-1})}} \Vert [\lambda \vert y_{\ve}\vert^\beta+\ve]\psi_{xxx}\Vert_{{ L^2(H^{1})}}\leq C, 
\end{align*}
where we used that $\Vert \psi\Vert_{L^2(H^4)\cap H^1(L^2)}\leq 1$.

Adding up, we arrive at $\sup_{\ve>0} \Vert \mu_{\ve}\Vert_{(L^2(H^4)\cap H^1(L^2))^*}\leq C$, i.e., $\lbrace \mu_{\ve}\rbrace$ is uniformly bounded with respect to $\ve>0$.

{\em Step 4:} By the bounds from the previous step, there exist $z_x^*\in L^2(H^{-1})$, and $\mu^*\in (L^2(H^4)\cap H^1(L^2))^*$ such that $z_{\ve,x}\rightharpoonup z_x^*$ weakly in $L^2(H^{-1})$, and $\mu_{\ve}\rightharpoonup \mu^*$ weakly in $(L^2(H^4)\cap H^1(L^2))^*$.

{\em Step 5:} Being equipped with the bounds and the convergence result in the last step, we may now show that $(z^*,\mu^*)$ solve \eqref{eq:thin_film_opt_cond_compl_cond}, \eqref{eq:thin_film_opt_cond_adjoint1}, and \eqref{eq:thin_film_opt_cond_opti_cond_u} by taking the limit in \eqref{eq:thin_film_opt_cond_eps_compl_cond}, \eqref{eq:thin_film_opt_cond_eps_adjoint1}, and \eqref{eq:thin_film_opt_cond_eps_opti_cond_u}. This concludes the proof.
\end{proof}

\section{Penalty approximation} \label{sec:thin_film_penalty}

A penalty approximation of the original problem is not clear due to 
the possible degeneracy of the original equation \eqref{eq:thin_film}, thus leaving open its well-posedness; this is the reason why we have introduced before the intermediate Problem \ref{prob:thin_film_opti_thin_film_eps} which involves the regularized equation \eqref{eq:thin_film_eps}. For
$\varepsilon >0$ fixed, Theorem~\ref{t5.3}
validates 
convergence of minimizers of Problem~\ref{prob:thin_film_opti_thin_film_eps_gamma} towards a minimum of Problem \ref{prob:thin_film_opti_thin_film_eps} for $\gamma \rightarrow 0$. Optimality conditions are then derived, which are the starting point for numerical studies in section~\ref{sec:thin_film_comp}. 

\begin{prob} \label{prob:thin_film_opti_thin_film_eps_gamma}
Let $\ve,\gamma>0$. We define the functional
\begin{align}
J_{\gamma}(y,u):= J(y,u) + \frac1{2\gamma} \int_0^T \int_{\Omega} \left\vert \big(C_0-y\big)^+\right\vert^2 \dd x \dd t. \label{eq:thin_film_penalty_functional}
\end{align}
Find $(y_{\gamma},u_{\gamma})$ as the minimum of $J_{\gamma}$ subject to \eqref{eq:thin_film_eps}.
\end{prob}
{The following result is immediate.}
\begin{thm}
There exists a solution $(y_{\gamma},u_{\gamma})$ of Problem \ref{prob:thin_film_opti_thin_film_eps_gamma}.
\end{thm}

The next theorem completes the overall convergence proof: The sequence of minima of the implementable Problem \ref{prob:thin_film_opti_thin_film_eps_gamma} converges to a minimum of Problem \ref{prob:thin_film_opti_thin_film_eps} for $\gamma\to 0$ (which again converges to a minimum of the original Problem \ref{prob:thin_film_opti}). { This result is valid for a certain range for involved parameters,
using $\widetilde{C}_0$ (see Lemma~\ref{lemma:thin_film_g0_gleich_0_positivity}) and $\widetilde{\varepsilon}_0$ (see Theorem \ref{thm:thin_film_existence_opti_regularization}).}

\begin{thm}\label{t5.3}
Let $\beta \geq 4$, $0 < C_0 \leq \widetilde{C}_0$, $0 < \ve \leq \widetilde{\varepsilon}_0$, and $\lbrace(y_{\gamma},u_{\gamma})\rbrace$ be a sequence of solutions of Problem \ref{prob:thin_film_opti_thin_film_eps_gamma}. Then, there exist $y^*\in L^2(H^4)\cap H^1(L^2)$, and $u^*\in L^2(H_0^1)$ such that $y_{\gamma}\rightharpoonup y^*$ weakly in $L^2(H^4)\cap H^1(L^2)$, and $u_{\gamma}\rightharpoonup u^*$ weakly in $L^2(H_0^1)$ for $\gamma\to 0$ (up to subsequences). Moreover, $(y^*,u^*)$ is a solution of Problem \ref{prob:thin_film_opti_thin_film_eps}.
\end{thm}
\begin{proof}
{\em Step 1:} Let $(y_{\gamma},u_{\gamma})$ be a solution of Problem \ref{prob:thin_film_opti_thin_film_eps_gamma}.
We first show that the functional is uniformly bounded with respect to $\gamma>0$: Let $(\bar y,\bar u)$ be the solution of Problem \ref{prob:thin_film_opti_thin_film_eps}, i.e., $(\bar y,\bar u)$ solve \eqref{eq:thin_film_eps}, $\bar y\geq C_0$ and $J(\bar y,\bar u)$ is the minimum value. Since $\bar y\geq C_0$, we have $J_{\gamma}(\bar y,\bar u)=J(\bar y,\bar u)$ independent of $\gamma>0$.

By the minimizing property of $(y_{\gamma},u_{\gamma})$, there holds 
\[ J_{\gamma}(y_{\gamma},u_{\gamma})\leq J_{\gamma}(\bar y,\bar u) = J(\bar y,\bar u) < \infty. \]
Hence, $J_{\gamma}(y_{\gamma},u_{\gamma})$ is uniformly bounded with respect to $\gamma>0$. 

{\em Step 2:} We want to get weak limit functions: from the definition of $J_{\gamma}$, we derive a uniform (with respect to $\gamma>0$) bound for $u_{\gamma}$ in the $L^2(H_0^1)$-norm. By the a-priori estimates from Lemma \ref{lemma:thin_film_apriori_h4}, $\{y_{\gamma}\}$ is uniformly (with respect to $\gamma>0$) bounded in the $L^2(H^4)\cap H^1(L^2)$-norm. Therefore, there exists $(y^*,u^*)\in \big(L^2(H^4)\cap H^1(L^2)\big)\times L^2(H_0^1)$ such that $(y_{\gamma},u_{\gamma})\rightharpoonup (y^*,u^*)$ weakly in the corresponding spaces (up to subsequences).

{\em Step 3:} We want to show that the limit functions $(y^*,u^*)$ are feasible for Problem \ref{prob:thin_film_opti_thin_film_eps}: it is easy to verify that $(y^*,u^*)$ solves \eqref{eq:thin_film_eps} like it was done, e.g., in the proof of Theorem \ref{thm:thin_film_existence_opti_mod}. It remains to show that $y^*\geq C_0$. Since $J_{\gamma}(y_{\gamma},u_{\gamma})\leq C$ uniformly in $\gamma>0$, we know that for $\gamma\to 0$,
\[ \int_0^T \int_{\Omega} \left\vert \big(C_0-y_{\gamma}\big)^+ \right\vert^2 \dd t \dd x \to 0, \]
i.e., we have $\big(C_0-y_{\gamma}\big)^+\to 0$ a.e. in $\Omega_T$, and hence $y^*\geq C_0$.

{\em Step 4:} Finally, we show that $(y^*,u^*)$ is a solution of Problem \ref{prob:thin_film_opti_thin_film_eps}: We have to show that $J(y^*,u^*)\leq J(y,u)$ for every $(y,u)$ solving \eqref{eq:thin_film_eps} and $y\geq C_0$. 

Let $(\bar y,\bar u)$ be a solution of Problem~\ref{prob:thin_film_opti_thin_film_eps}. By the first parts of the proof, we know that $(y^*,u^*)$ is feasible for Problem~\ref{prob:thin_film_opti_thin_film_eps}, i.e., we have $J_{\gamma}(y^*,u^*)=J(y^*,u^*)$. Since $(y_{\gamma},u_{\gamma})\rightharpoonup (y^*,u^*)$ weakly in the corresponding spaces by the second part of the proof, and $J$ is weakly lower semi-continuous, we have
\begin{align}
J(y^*,u^*)\leq \liminf_{\gamma\to 0} J_{\gamma}(y_{\gamma},u_{\gamma}) \leq J(\bar y,\bar u), \label{eq:thin_film_pen_convergence_minimum_eigenschaft}
\end{align}
where we used Step 1 which relies on $(y_{\gamma},u_{\gamma})$ being a solution of Problem~\ref{prob:thin_film_opti_thin_film_eps_gamma}.

Since $(\bar y,\bar u)$ is a minimum of $J$, all quantities in \eqref{eq:thin_film_pen_convergence_minimum_eigenschaft} must be equal, i.e., $(y^*,u^*)$ is a solution of Problem \ref{prob:thin_film_opti_thin_film_eps}.
\end{proof}

As in the last sections, we can now derive an analogon to \eqref{eq:thin_film_opt_cond} and \eqref{eq:thin_film_opt_cond_eps}, respectively, which can be proven even by the standard Lagrange multiplier theorem due to the absence of state constraints.

\begin{thm}
Let $(y,u)$ be a minimum of Problem \ref{prob:thin_film_opti_thin_film_eps_gamma}. Then, there exists a Lagrange multiplier $z\in L^2(L^2)$ such that the following equations are fulfilled.
\begin{subequations}
\begin{align}
y_t ={}& -\big([\lambda\vert y\vert^\beta+\ve]y_{xxx}\big)_x + u_x, \label{eq:thin_film_opt_cond_eps_gamma_state} \\
0 ={}& \langle y-\tilde{y},\varphi\rangle + \big\langle z,\varphi_t + (\beta\lambda y^{\beta-1}y_{xxx}\varphi)_x\big\rangle + \big\langle z,\big( [\lambda\vert y\vert^\beta+\ve]\varphi_{xxx}\big)_x\big\rangle \label{eq:thin_film_opt_cond_eps_gamma_adjoint1} \\
&+ \frac1{\gamma}\langle \varphi,(C_0-y)^+\mu\rangle \quad \forall \varphi\in X, \nonumber \\
0 ={}& -\alpha u_{xx} + z_x, \label{eq:thin_film_opt_cond_eps_opti_gamma_cond_u}
\end{align} \label{eq:thin_film_opt_cond_eps_gamma}%
\end{subequations}
together with initial conditions $y(0,.)=y_0$, $z(T,.)=0$, and boundary conditions $y_x=y_{xxx}=z_x=z_{xxx}=0$ in $a,b$.
\end{thm}

\section{Computational studies} \label{sec:thin_film_comp}

{We discretize the optimization Problem~\ref{prob:thin_film_opti_thin_film_eps_gamma} in space and time and
follow the strategy `first discretize, then optimize' (see e.g.~\cite{hinze_buch}) to arrive at the finite dimensional Problem~\ref{prob:thin_film_opti_disc}.}
The studies
which are reported below are meant to do both, illustrate and complement the theoretical results in the previous sections.

\subsection{Discretization of the equation}

We use the following space-time discretization scheme for \eqref{eq:thin_film_eps}, which was originally suggested for \eqref{eq:thin_film} in \cite{becker2005}. 

Let $hN_{\text{space}}=b-a$ and $x_i:=a+ih$ for $i=0,\ldots,N_{\text{space}}$ denote the set of spatial nodes. Define the standard finite element space $V_h$ containing piecewise linear functions, via
\[ V_h:=\left\lbrace v_h\in\Con([a,b]): \; v_h\big\vert_{[x_i,x_{i+1}]}\in P_1 \right\rbrace, \]
cf.~\cite{brenner_scott}. The function $P_h:L^2\to V_h$ denotes the projection onto $V_h$ with respect to the $L^2$ scalar product.

Let $kN_{\text{time}}=T$, and let $t_n:=nk$ for $n=0,\ldots,N_{\text{time}}$ denote the nodal points of a time grid which covers $[0,T]$.
In below, let $\lbrace V^n\rbrace \subseteq X_h$ denote a family of finite element functions evaluated at subsequent times $t_n$, while $V:\Omega_T\to \R$ stands for the piecewise affine, globally continuous time interpolant of $\lbrace V^n\rbrace$. Sometimes, we also write $V(t=t_n)$ instead of $V^n$.

The discrete version of \eqref{eq:thin_film_eps}  reads as follows.

\begin{prob} \label{prob:thin_film_disc}
Let $Y_0:= P_hy_0\in V_h$. Set $Y^0:=Y_0$, find $P^0\in V_h$ such that
\[ (Y^0_x,\Phi_x) - (P^0,\Phi) = 0 \quad \forall \Phi\in V_h. \]
Then for $n=1,\ldots,N_{\text{time}}-1$ find $Y^{n+1}\in V_h$, $P^{n+1}\in V_h$ and $P^{n+1}\in V_h$, such that
\begin{align}
\begin{aligned}
\frac1k(Y^{n+1}-Y^n,\Phi) + ([\lambda \vert Y^{n+1}\vert^\beta + \ve]P^{n+1}_x,\Phi_x) &= (U_x(t_{n+1}),\Phi) \quad \forall \Phi\in V_h, \\
(Y^{n+1}_x,\Phi_x) - (P^{n+1},\Phi) &= 0 \quad \forall \Phi\in V_h.
\end{aligned} \label{eq:thin_film_eps_disc}
\end{align}
\end{prob}

The coupled system \eqref{eq:thin_film_eps_disc} is solved by Newton's method with exact derivatives; all terms (which are polynomials of higher order) are assembled exactly by using an accurate quadrature rule. 

Lemma~\ref{lemma:thin_film_eps_existence} motivates solvability of \eqref{eq:thin_film_eps_disc} for $\ve>0$. However, for small $\ve>0$, the system matrix has a high condition number in the presence of related large values of the approximation of $U_x(t_n)$ and small values of $\lbrace Y^n\rbrace$ due to the algebraic form of $f_{\ve}$. We encountered this problem in the form of a singular system matrix on the level of numerical linear algebra. Smaller values of $k$, larger values of $\ve$ and---in the context of optimal control---state constraints help to overcome this issue.

{For all experiments in this section, we choose $\lambda=1.0$ and $\beta = 3$}. The iteration in Newton's method stops if the difference of two consecutive iterations is less than $10^{-10}$, or if the maximum number of iterations exceeds $1\,000$. However, except for those experiments with singular system matrices, the observed number of iterates was well below
these values (average 2--5/max.~30 iterations; highly depending on the specific experiment).

\subsection{Simulations of the equation}

We want to find a right-hand side $U$ such that the corresponding solution is non-positive in order to show the need of state constraints for the optimization. 

For the first experiment, we take $[a,b]=[0,5]$, $T=1.0$, $N_{\text{space}}=48$, $N_{\text{time}}=30\,000$. We solve \eqref{eq:thin_film_eps_disc} for $U = 0$ and $U(x)=0.35\sin\left(\frac{\pi x}{b-a}\right)$ and $\ve=0.03$. For comparison, we also include the solution $Y$ for $U = 0$ and $\ve=0$; see Figure~\ref{fig:thin_film_negative_sol_only_state_bsp}. We see that the solution $Y$ becomes significantly negative for $U\neq 0$, while the solution $Y$ stays positive at all times no matter how $\ve \geq 0$ { is chosen for $U = 0$} (hence the negativity effect does not depend on $\ve$).

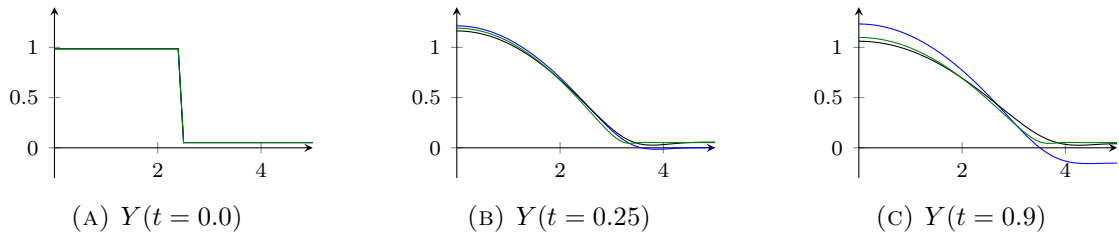
\begin{figure}[ht]
\centering
\begin{subfigure}[b]{0.32\textwidth}
\centering
\begin{tikzpicture}
\begin{axis}[three_style,
ymin=-0.3,
ymax=1.4,
]
\addplot[blue] table {figures/fig01/fig01_data3_0.dat}; 
\label{pgfplots:thin_film_bsp_neg_sol}
\addplot[black] table {figures/fig01/fig01_data1_0.dat}; 
\label{pgfplots:thin_film_bsp_neg_sol_vgl_u0_eps003}
\addplot[green!50!black] table {figures/fig01/fig01_data2_0.dat}; 
\label{pgfplots:thin_film_bsp_neg_sol_vgl_u0_eps0}
\end{axis}
\end{tikzpicture}
\caption{$Y(t=0.0)$}
\end{subfigure}
\hfill
\begin{subfigure}[b]{0.32\textwidth}
\centering
\begin{tikzpicture}
\begin{axis}[three_style,
ymin=-0.3,
ymax=1.4,
]
\addplot[blue] table {figures/fig01/fig01_data3_50.dat}; 
\addplot[black] table {figures/fig01/fig01_data1_50.dat}; 
\addplot[green!50!black] table {figures/fig01/fig01_data2_50.dat}; 
\end{axis}
\end{tikzpicture}
\caption{$Y(t=0.25)$}
\end{subfigure}
\hfill
\begin{subfigure}[b]{0.32\textwidth}
\centering
\begin{tikzpicture}
\begin{axis}[three_style,
ymin=-0.3,
ymax=1.4,
]
\addplot[blue] table {figures/fig01/fig01_data3_180.dat}; 
\addplot[black] table {figures/fig01/fig01_data1_180.dat}; 
\addplot[green!50!black] table {figures/fig01/fig01_data2_180.dat}; 
\end{axis}
\end{tikzpicture}
\caption{$Y(t=0.9)$}
\end{subfigure}
\caption{Solution $Y$ at different times for a given right-hand side $U\neq 0$ and $\ve=0.03$~(\ref{pgfplots:thin_film_bsp_neg_sol}), for $U = 0$ and $\ve=0.03$~(\ref{pgfplots:thin_film_bsp_neg_sol_vgl_u0_eps003}), and for $U = 0$ and $\ve=0$~(\ref{pgfplots:thin_film_bsp_neg_sol_vgl_u0_eps0}).} \label{fig:thin_film_negative_sol_only_state_bsp}
\end{figure}

\subsection{Discretization of the optimal control problem}

We use the following discrete version of Problem~\ref{prob:thin_film_opti_thin_film_eps_gamma} for the simulations.

\begin{prob} \label{prob:thin_film_opti_disc}
Let $\ve>0$, $\gamma\geq 0$, and $k >0$. Define $J_{\gamma,\text{\rm disc}}:V_h^{N_{\text{time}}+1}\times V_h^{N_{\text{time}}+1}\to \R$ via
\[ J_{\gamma,\text{\rm disc}}(Y,U):= \frac{k}2 \sum_{n=0}^{N_{\text{\rm time}}} \Vert Y^n-\tilde{Y}^n\Vert^2 + \frac{\alpha k}2 \sum_{n=0}^{N_{\text{\rm time}}} \Vert U_x^n\Vert^2 + \frac{k}{2\gamma}\sum_{n=0}^{N_{\text{\rm time}}} \Vert (C_0-Y^n)^+\Vert^2, \]
where the last term is ignored if we set $\gamma=0$, and $\Vert.\Vert$ here stands for the Euclidean norm. For $\tilde{Y}^n\notin V_h$, we use the interpolation of it.

Find $(Y,U)$ which minimizes $J_{\gamma,\text{\rm disc}}$ subject to \eqref{eq:thin_film_eps_disc}.
\end{prob}
The corresponding optimality conditions for $(Y,U,Z,S) \in [V_h^{N_{\rm time}}]^4$ are for all $0 \leq n \leq N_{{\rm time}}-1$: 
\begin{subequations}
\begin{align}
\frac1k(Y^{n+1}-Y^n,\Phi) + (\lambda\vert Y^{n+1}\vert^\beta P^{n+1}_x,\Phi_x) ={}& (U_x(t_{n+1}),\Phi) \quad \forall \Phi\in V_h, \label{eq:thin_film_opt_cond_disc_state} \\
(Y^{n+1}_x,\Phi_x) - (P^{n+1},\Phi) ={}& 0 \quad \forall \Phi\in V_h, \label{eq:thin_film_opt_cond_disc_costate} \\
\frac1k(\Phi,Z^n) + (\beta \lambda \vert Y^{n+1}\vert^{\beta-1}\Phi P^{n+1}_x,Z^n_x) + (\Phi_x,S_x)={}& \frac1k(\Phi,Z^{n+1})+ (\Phi,\tilde{Y}^{n+1}-Y^{n+1}) \label{eq:thin_film_opt_cond_disc_adjoint} \\
&+ \frac1{\gamma}\left(\Phi,(C_0-Y^{n+1})^+\right) \quad \forall \Phi\in V_h, \nonumber \\
(\lambda \vert Y^{n+1}\vert^\beta\Phi_x,Z^n) -(\Phi,S^n) ={}& 0 \quad \forall \Phi\in V_h, \label{eq:thin_film_opt_cond_disc_coadjoint} \\
\alpha (U_x,\Phi_x) + (Z_x,\Phi) ={}& 0 \quad \forall \Phi\in V_h, \label{eq:thin_film_opt_cond_disc_cond_u}
\end{align} \label{eq:thin_film_opt_cond_disc}%
\end{subequations}
together with initial conditions $Y^0=Y_0$, $Z^{N_{\text{time}}}=0$. Conditions \eqref{eq:thin_film_opt_cond_disc_costate}, \eqref{eq:thin_film_opt_cond_disc_coadjoint}, and \eqref{eq:thin_film_opt_cond_disc_cond_u} are also valid for $n=0$.

Weak solutions of \eqref{eq:thin_film_eps} are unique, which is also valid for the discrete version \eqref{eq:thin_film_eps_disc} of it for $k>0$ small enough; hence the operator $U\mapsto Y(U)$ is well-defined. Therefore, we can use a steepest descent algorithm with Armijo step size rule to numerically solve Problem~\ref{prob:thin_film_opti_disc}; see \cite{herzog_kunisch,hinze_buch}.
The corresponding algorithm we use reads as follows.

\begin{alg} \label{alg:thin_film_grad_verfahren_armijo}
Set $U_0 = 0$ and fix $\sigma_*>0$, $0<\beta<1$, $\delta_{\text{\rm tol}}>0$. Compute $(Y_1,P_1)$ from solving \eqref{eq:thin_film_eps_disc}, then compute $(Z_1,S_1)$ from solving \eqref{eq:thin_film_opt_cond_disc_adjoint} and \eqref{eq:thin_film_opt_cond_disc_coadjoint}. Repeat for $r\geq 0$:
\begin{enumerate}
\item Evaluate $\nabla \tilde{J}(U_r)=\alpha U_r + (Z_r)_x$ and evaluate $\tilde{J}(U_r)$.
\item Repeat for $s\geq 0$:
\begin{enumerate}
\item Define $U_{r+1}^{(s)}:=U_r - \beta^s\nabla \tilde{J}(U_r)$.
\item Compute $(Y_{r+1}^{(s)},P_{r+1}^{(s)})$ from solving \eqref{eq:thin_film_eps_disc} for $U_{r+1}^{(s)}$ as right-hand side.
\item STOP, if
\begin{align}
\tilde{J}(U_{r+1}^{(s)})-\tilde{J}(U_r) \leq -\sigma_* \beta^s\Vert\nabla\tilde{J}(U_r)\Vert^2, \label{eq:thin_film_armijo_step_size_ok}
\end{align}
and set $U_{r+1}:=U_{r+1}^{(s)}$.
\end{enumerate}
\item Compute $(Z_{r+1},S_{r+1})$ from solving \eqref{eq:thin_film_opt_cond_disc_adjoint} and \eqref{eq:thin_film_opt_cond_disc_coadjoint}.
\item STOP, if $\Vert \nabla \tilde{J}(U_{r+1})\Vert^2\leq \delta_{\text{\rm tol}}$ and set $U_{\text{\rm opt}}=U_{r+1}$, $Y_{\text{\rm opt}}=Y_{r+1}$.
\end{enumerate}
\end{alg}
{Here, $\nabla \tilde{J}(\cdot)$ corresponds to the 
finite dimensional version of the gradient of $\tilde{J}$, which is the left-hand side of (\ref{eq:thin_film_opt_cond_disc_cond_u})}.
In all the studies below, we set $\sigma_*:=10^{-5}$ and $\beta:=0.15$. The stopping condition is set to be $\delta_{\text{tol}}:=5\cdot 10^{-5}$, which is obtained after $700$ up to $50\,000$ iterations. The number of iterations highly depends on the given data (i.e., on $Y_0$, $\tilde{Y}$, and on $\alpha,\ve,\gamma>0$).

\subsection{Comparison of the parameter \texorpdfstring{$\ve$}{epsilon}}
Let $[a,b]=[0,5]$, $T=1.0$, $N_{\text{space}}=30$, $N_{\text{time}}=5\,000$, and $\alpha=10^{-7}$. We solve \eqref{eq:thin_film_eps_disc} for $U = 0$ to study the dependencies on $\ve>0$; see Figure~\ref{fig:thin_film_comp_vgl_eps_no_control}. The bigger the value of $\ve$, the more dissipative is the evolution, and the solution becomes almost flat after a short time. In contrast to this, the solution needs a longer period of time to approach a flat profile for small values of $\ve$.

For a large value of $\ve$, the solution is slightly negative in some regions. This is due to the fact that there is no maximum principle for the biharmonic problem, which would enforce the solution to stay positive. This effect vanishes for decreasing values of $\ve$, which is in agreement with Lemma~\ref{lemma:thin_film_g0_gleich_0_positivity}.

\begin{figure}[ht]
\centering
\begin{subfigure}[b]{0.32\textwidth}
\begin{tikzpicture}
\begin{axis}[three_style,
ymin=-0.2,
ymax=1.8,
]
\addplot[blue] table {figures/fig02/01/data0.dat}; 
\label{pgfplots:thin_film_eps_no_control_eps05}
\addplot[black] table {figures/fig02/02/data0.dat}; 
\label{pgfplots:thin_film_eps_no_control_eps005}
\end{axis}
\end{tikzpicture}
\caption{$Y(t=0.0)$}
\end{subfigure}
\hfill
\begin{subfigure}[b]{0.32\textwidth}
\begin{tikzpicture}
\begin{axis}[three_style,
ymin=-0.2,
ymax=1.8,
]
\addplot[blue] table {figures/fig02/01/data20.dat}; 
\addplot[black] table {figures/fig02/02/data20.dat}; 
\end{axis}
\end{tikzpicture}
\caption{$Y(t=0.1)$}
\end{subfigure} \hfill
\begin{subfigure}[b]{0.32\textwidth}
\begin{tikzpicture}
\begin{axis}[three_style,
ymin=-0.2,
ymax=1.8,
]
\addplot[blue] table {figures/fig02/01/data40.dat}; 
\addplot[black] table {figures/fig02/02/data40.dat}; 
\end{axis}
\end{tikzpicture}
\caption{$Y(t=0.2)$}
\end{subfigure} \\
\begin{subfigure}[b]{0.32\textwidth}
\begin{tikzpicture}
\begin{axis}[three_style,
ymin=-0.2,
ymax=1.8,
]
\addplot[blue] table {figures/fig02/01/data100.dat}; 
\addplot[black] table {figures/fig02/02/data100.dat}; 
\end{axis}
\end{tikzpicture}
\caption{$Y(t=0.5)$}
\end{subfigure}
\hfill
\begin{subfigure}[b]{0.32\textwidth}
\begin{tikzpicture}
\begin{axis}[three_style,
ymin=-0.2,
ymax=1.8,
]
\addplot[blue] table {figures/fig02/01/data150.dat}; 
\addplot[black] table {figures/fig02/02/data150.dat}; 
\end{axis}
\end{tikzpicture}
\caption{$Y(t=0.75)$}
\end{subfigure}
\hfill
\begin{subfigure}[b]{0.32\textwidth}
\begin{tikzpicture}
\begin{axis}[three_style,
ymin=-0.2,
ymax=1.8,
]
\addplot[blue] table {figures/fig02/01/data200.dat}; 
\addplot[black] table {figures/fig02/02/data200.dat}; 
\end{axis}
\end{tikzpicture}
\caption{$Y(t=1.0)$} 
\end{subfigure} 
\caption{Solution $Y$ of \eqref{eq:thin_film_eps_disc} for $U = 0$, and for $\ve=0.5$~(\ref{pgfplots:thin_film_eps_no_control_eps05}) and $\ve=0.05$~(\ref{pgfplots:thin_film_eps_no_control_eps005}) at different times.} \label{fig:thin_film_comp_vgl_eps_no_control}
\end{figure}
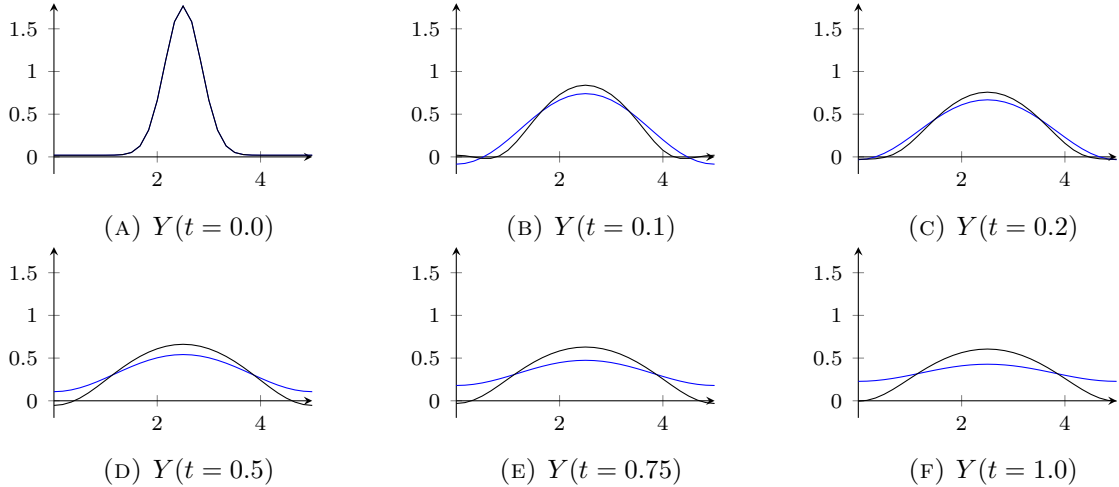

We repeat the above experiment with the same parameters for Problem~\ref{prob:thin_film_opti_disc} where $\gamma = 0$; see Figure~\ref{fig:thin_film_comp_vgl_eps_with_control}. In contrast to the previous experiment from Figure~\ref{fig:thin_film_comp_vgl_eps_no_control}, there is no significant difference between the computed evolution of the optimal states for varying values of $\ve$. This is due to the fact that the optimal state $Y=Y(\ve)$ belongs to different optimal controls $U=U(\ve)$ which drive the solution to approximately attain the given target profile $\tilde{Y}$. Figure~\ref{fig:thin_film_comp_vgl_eps_with_control} shows that relevant controls are active, as opposed to Figure~\ref{fig:thin_film_comp_vgl_eps_no_control} where $U = 0$.

\begin{figure}[ht]
\centering
\begin{subfigure}[b]{0.32\textwidth}
\begin{tikzpicture}
\begin{axis}[three_style,
ymin=-0.05,
ymax=1.8,
]
\addplot[blue] table {figures/fig03/01/data0.dat}; 
\label{pgfplots:thin_film_eps_with_control_eps05}
\addplot[black] table {figures/fig03/02/data0.dat}; 
\label{pgfplots:thin_film_eps_with_control_eps005}
\addplot[red,thick,dashed] table {figures/fig03/y_tilde.dat}; 
\label{pgfplots:thin_film_eps_with_control_target}
\end{axis}
\end{tikzpicture}
\caption{$Y(t=0.0)$}
\end{subfigure}
\hfill
\begin{subfigure}[b]{0.32\textwidth}
\begin{tikzpicture}
\begin{axis}[three_style,
ymin=-0.05,
ymax=1.8,
]
\addplot[blue] table {figures/fig03/01/data20.dat}; 
\addplot[black] table {figures/fig03/02/data20.dat}; 
\addplot[red,thick,dashed] table {figures/fig03/y_tilde.dat}; 
\end{axis}
\end{tikzpicture}
\caption{$Y(t=0.1)$}
\end{subfigure}
\hfill
\begin{subfigure}[b]{0.32\textwidth}
\begin{tikzpicture}
\begin{axis}[three_style,
ymin=-0.05,
ymax=1.8,
]
\addplot[blue] table {figures/fig03/01/data40.dat}; 
\addplot[black] table {figures/fig03/02/data40.dat}; 
\addplot[red,thick,dashed] table {figures/fig03/y_tilde.dat}; 
\end{axis}
\end{tikzpicture}
\caption{$Y(t=0.2)$}
\end{subfigure} \\
\begin{subfigure}[b]{0.32\textwidth}
\begin{tikzpicture}
\begin{axis}[three_style,
ymin=-0.05,
ymax=1.8,
]
\addplot[blue] table {figures/fig03/01/data100.dat}; 
\addplot[black] table {figures/fig03/02/data100.dat}; 
\addplot[red,thick,dashed] table {figures/fig03/y_tilde.dat}; 
\end{axis}
\end{tikzpicture}
\caption{$Y(t=0.5)$}
\end{subfigure}
\hfill
\begin{subfigure}[b]{0.32\textwidth}
\begin{tikzpicture}
\begin{axis}[three_style,
ymin=-0.05,
ymax=1.8,
]
\addplot[blue] table {figures/fig03/01/data150.dat}; 
\addplot[black] table {figures/fig03/02/data150.dat}; 
\addplot[red,thick,dashed] table {figures/fig03/y_tilde.dat}; 
\end{axis}
\end{tikzpicture}
\caption{$Y(t=0.75)$}
\end{subfigure}
\hfill
\begin{subfigure}[b]{0.32\textwidth}
\begin{tikzpicture}
\begin{axis}[three_style,
ymin=-0.05,
ymax=1.8,
]
\addplot[blue] table {figures/fig03/01/data200.dat}; 
\addplot[black] table {figures/fig03/02/data200.dat}; 
\addplot[red,thick,dashed] table {figures/fig03/y_tilde.dat}; 
\label{pgfplots:thin_film_alpha_state_y_tilde}
\end{axis}
\end{tikzpicture}
\caption{$Y(t=1.0)$}
\end{subfigure}
\caption{Target $\tilde{Y}$~(\ref{pgfplots:thin_film_alpha_state_y_tilde}), and optimal state $Y$ for $\ve=0.5$~(\ref{pgfplots:thin_film_eps_with_control_eps05}) and $\ve=0.05$~(\ref{pgfplots:thin_film_eps_with_control_eps005}) at different times.} \label{fig:thin_film_comp_vgl_eps_with_control}
\end{figure}
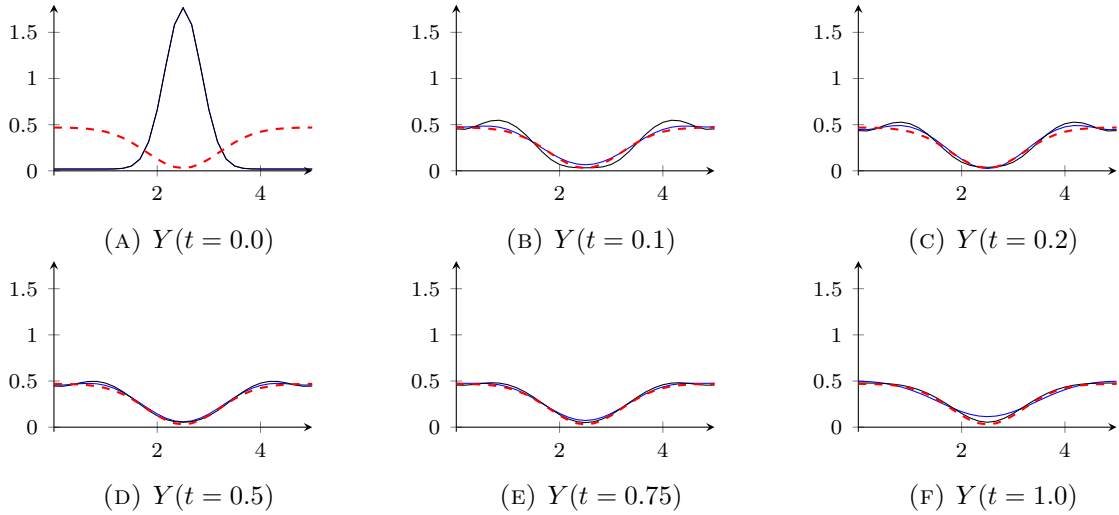

\subsection{Comparison of the parameter \texorpdfstring{$\gamma$}{gamma} and dewetting application} \label{subsec:thin_film_comp_gamma}

In this experiment, we take $C_0=0.0$, $\alpha=10^{-7}$, $\ve=0.1$, $N_{\text{space}}=42$, $N_{\text{time}}=5\,000$ and simulate for different values of $\gamma>0$; see Figure~\ref{fig:thin_film_comp_vgl_gamma}. Here, $\tilde{Y}$ is constant in time and the profile is given in the figure. 
The practical importance of a thin film with locally vanishing height is discussed in
\cite{becker_nature}, and is known as dewetting effect.  Figure~\ref{fig:thin_film_comp_vgl_gamma} displays snapshots
for the height function of a corresponding study. We observe a different dynamics for $\gamma > 0$ (which enforces non-negativity) vs.~$\gamma =0$ (where violation of non-negativity happens).

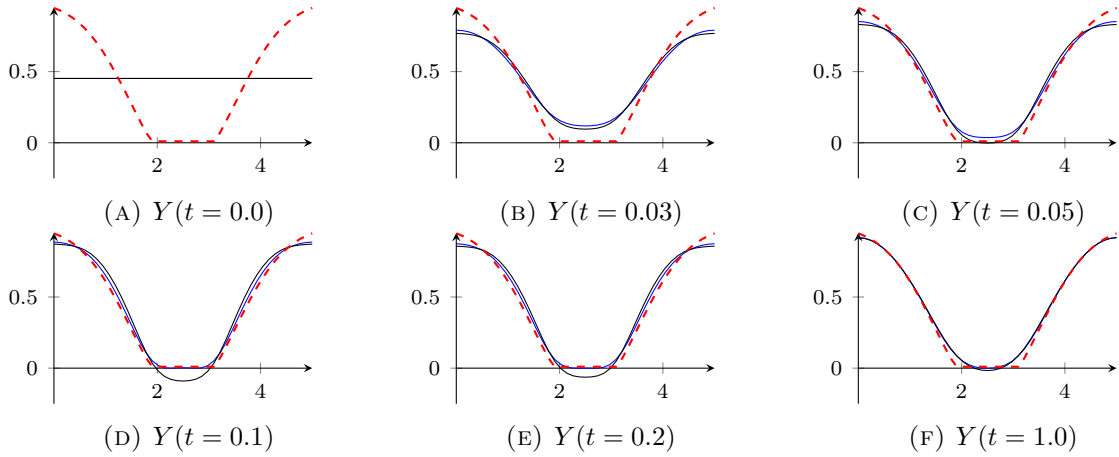
\begin{figure}
\begin{subfigure}[b]{0.32\textwidth}
\begin{tikzpicture}
\begin{axis}[three_style,
ymin=-0.25,
ymax=0.95,
]
\addplot[blue] table {figures/fig06/fig06_data1_0.dat}; 
\label{pgfplots:thin_film_gamma_ex4c_state_gamma0}
\addplot[black] table {figures/fig06/fig06_data2_0.dat}; 
\label{pgfplots:thin_film_gamma_ex4c_state_gamma002}
\addplot[red,thick,dashed] table {figures/fig06/fig06_y_tilde.dat};
\label{pgfplots:thin_film_gamma_ex4c_target}
\end{axis}
\end{tikzpicture}
\caption{$Y(t=0.0$)}
\end{subfigure}
\hfill
\begin{subfigure}[b]{0.32\textwidth}
\begin{tikzpicture}
\begin{axis}[three_style,
ymin=-0.25,
ymax=0.95,
]
\addplot[blue] table {figures/fig06/fig06_data1_15.dat}; 
\addplot[black] table {figures/fig06/fig06_data2_15.dat}; 
\addplot[red,thick,dashed] table {figures/fig06/fig06_y_tilde.dat};
\end{axis}
\end{tikzpicture}
\caption{$Y(t=0.03$)}
\end{subfigure}
\hfill
\begin{subfigure}[b]{0.32\textwidth}
\begin{tikzpicture}
\begin{axis}[three_style,
ymin=-0.25,
ymax=0.95,
]
\addplot[blue] table {figures/fig06/fig06_data1_24.dat}; 
\addplot[black] table {figures/fig06/fig06_data2_24.dat}; 
\addplot[red,thick,dashed] table {figures/fig06/fig06_y_tilde.dat};
\end{axis}
\end{tikzpicture}
\caption{$Y(t=0.05$)}
\end{subfigure} \\
\begin{subfigure}[b]{0.32\textwidth}
\begin{tikzpicture}
\begin{axis}[three_style,
ymin=-0.25,
ymax=0.95,
]
\addplot[blue] table {figures/fig06/fig06_data1_50.dat}; 
\addplot[black] table {figures/fig06/fig06_data2_50.dat}; 
\addplot[red,thick,dashed] table {figures/fig06/fig06_y_tilde.dat};
\end{axis}
\end{tikzpicture}
\caption{$Y(t=0.1$)}
\end{subfigure}
\hfill
\begin{subfigure}[b]{0.32\textwidth}
\begin{tikzpicture}
\begin{axis}[three_style,
ymin=-0.25,
ymax=0.95,
]
\addplot[blue] table {figures/fig06/fig06_data1_100.dat}; 
\addplot[black] table {figures/fig06/fig06_data2_100.dat}; 
\addplot[red,thick,dashed] table {figures/fig06/fig06_y_tilde.dat};
\end{axis}
\end{tikzpicture}
\caption{$Y(t=0.2$)}
\end{subfigure}
\hfill
\begin{subfigure}[b]{0.32\textwidth}
\begin{tikzpicture}
\begin{axis}[three_style,
ymin=-0.25,
ymax=0.95,
]
\addplot[blue] table {figures/fig06/fig06_data1_500.dat}; 
\addplot[black] table {figures/fig06/fig06_data2_500.dat}; 
\addplot[red,thick,dashed] table {figures/fig06/fig06_y_tilde.dat};
\end{axis}
\end{tikzpicture}
\caption{$Y(t=1.0$)}
\end{subfigure}
\caption{Target $\tilde{Y}$~(\ref{pgfplots:thin_film_gamma_ex4c_target}) and optimal states $Y$ for $\gamma =0.02$~(\ref{pgfplots:thin_film_gamma_ex4c_state_gamma0}) and $\gamma=0$~(\ref{pgfplots:thin_film_gamma_ex4c_state_gamma002}) at different times.} \label{fig:thin_film_comp_vgl_gamma}
\end{figure}

\printbibliography

\end{document}